\documentclass[11pt,namelimits,sumlimits,a4paper]{amsart}
\usepackage[utf8]{inputenc}
\usepackage{hyperref}
\usepackage{subfiles}
\usepackage{amssymb,amsmath}
\usepackage[mathscr]{eucal}
\usepackage{slashed}

\usepackage[T1]{fontenc}
\usepackage[UKenglish]{babel}
\usepackage{amsfonts}
\usepackage{fancyhdr}
\usepackage{graphicx}
\usepackage{amsmath}
\usepackage{amsthm}
\usepackage{amsmath,amscd}
\usepackage{latexsym}
\usepackage{cite}
\usepackage{amssymb,amsmath}
\usepackage{tensor}
\usepackage{multicol}
\usepackage{comment}
\usepackage[mathscr]{eucal}
\usepackage{slashed}
\usepackage{times,psfrag}
\usepackage{mathtools}
\usepackage{hyperref}
\usepackage{multirow}
\usepackage{multicol}
\usepackage{longtable}
\usepackage{bm}
\usepackage{fancyhdr}
\usepackage{float}
\usepackage{esint}
\usepackage{layout}
\usepackage{enumitem}
\usepackage{lmodern}
\usepackage[rgb]{xcolor}
\usepackage{tikz}
\usetikzlibrary{shapes,arrows}
\usetikzlibrary{matrix,arrows}
\usepackage{dsfont}
\usepackage{tikz-cd}
\usepackage{stmaryrd}

\newlength{\dhatheight}

\newcommand{\lie}{\mathcal{L}}
\newcommand{\SL}{\operatorname{SL}}
\newcommand{\dvol}{\operatorname{dvol}}

\newcommand{\J}{\mathcal{J}}

\newcommand{\aut}{\mathfrak{aut}}
\newcommand{\R}{\mathbb{R}}

\newcommand{\C}{\mathbb{C}}

\newcommand{\Z}{\mathbb{Z}}
\newcommand{\g}{\mathfrak{g}}

\newcommand{\SO}{\mathrm{SO}}
\newcommand{\SU}{\mathrm{SU}}
\newcommand{\U}{\mathrm{U}}
\newcommand{\Sl}{\mathrm{SL}}
\newcommand{\st}{\big|\;}

\newcommand{\Sp}{\mathrm{Sp}}
\renewcommand{\dim}{\mathrm{dim}}

\newcommand{\GL}{\operatorname{GL}}

\newcommand{\vol}{\operatorname{vol}}
\newcommand{\abs}[1]{|#1|}
\newcommand{\norm}[1]{||#1||}
\newcommand{\Ind}{\operatorname{Ind}}
\renewcommand{\div}{\operatorname{div}}

\newcommand{\curl}{\mathrm{curl}}
\newcommand{\Sym}{\operatorname{Sym}}

\newcommand{\Spin}{\operatorname{Spin}}

\renewcommand{\phi}{\varphi}
\newtheorem{theorem}{Theorem}[section]
\newtheorem*{theorem*}{Theorem}

\newtheorem*{question*}{Question}
\newtheorem*{conjecture*}{Conjecture}
\newtheorem*{proposition*}{Proposition}
\newtheorem{remark}[theorem]{Remark}
\newtheorem{example}[theorem]{Example}
\newtheorem{definition}[theorem]{Definition}
\newtheorem{corollary}[theorem]{Corollary}
\newtheorem{proposition}[theorem]{Proposition}

\newtheorem{lemma}[theorem]{Lemma}

\begin{document}
\author{Enric Solé-Farré}
\title{Stability of nearly Kähler and nearly parallel $G_2$-manifolds}
\address{Department of Mathematics, University College London, London WC1E 6BT, United Kingdom}
\thanks{{\tt enric.sole-farre.21@ucl.ac.uk}}
\begin{abstract}
We investigate generalisations of Hitchin’s functionals, whose critical points correspond to nearly Kähler and nearly parallel $G_2$-structures. We focus on the gradient flow of these functionals and the spectral decomposition of their Hessians with respect to natural indefinite inner products.

We introduce a Morse-like index for these functionals, termed the Hitchin index. We prove this index provides a lower bound for the Einstein co-index and explore its relationship with the deformation theory of $G_2$ and $\Spin(7)$-conifolds.
\end{abstract}
\maketitle
\vspace{-1cm}
\section{Introduction}

Given a complete Riemannian manifold $(M^{n-1},g)$, one can study its associated metric cone $C(M)= (M\times \R_+, dr^2 +r^2g)$, and the interplay between structures on $M$ and its cone. For instance, the cone metric will be Ricci-flat if and only if $g$ is Einstein with constant $(n-2)$. In terms of holonomy, one can ask what structure $M$ carries when the cone has holonomy contained in $\U(n/2)$, $\SU(n/2)$, $\Sp(n/4)$, $G_2$ or $\Spin(7)$. The corresponding geometries on the link are called Sasaki, Sasaki-Einstein, 3-Sasaki, nearly Kähler and nearly parallel $G_2$, respectively. We will focus on the latter two. Since Riemannian manifolds with holonomy contained in either $G_2$ or $\Spin(7)$ are Ricci flat, nearly Kähler and nearly parallel $G_2$-manifolds are Einstein with constant $\lambda=5$ and $\lambda=6$ respectively. 

The term nearly Kähler structures here corresponds to what is often referred to as almost Hermitian strict nearly Kähler manifolds in dimension $6$ or Gray manifolds in the literature. Gray \cite{Gray70} introduced the general term nearly Kähler to distinguish a particular class of almost Hermitian manifolds in every even dimension. Strict nearly Kähler manifolds of dimension $6$ are of particular interest due to the work of Nagy \cite{Nagy02}, who showed that every nearly Kähler manifold is locally isometric to a Riemannian product of 6-dimensional strict nearly Kähler manifolds, nearly Kähler homogeneous spaces and twistor spaces over positive scalar curvature quaternionic-Kähler manifolds.

Only six examples of simply connected nearly Kähler manifolds are currently known. Four of them are homogeneous: 
\begin{multicols}{2}
\begin{itemize}
    \item $(S^6, g_{round})= G_2/ \SU(3)$, 
    \item $\C P^3 = \Sp(2)/\U(1)\times \Sp(1)$ and 
    \item[\vspace{\fill}]
    \item $S^3 \times S^3 = \SU(2)^3/\triangle \SU(2)$,
    \item $\mathbb{F}_{1,2} = \SU(3)/T^2$.
\end{itemize}
\end{multicols}
\vspace{-2em}
In 2016, Foscolo and Haskins \cite{FH17} produced two new examples of nearly Kähler structures on $S^6$ and $S^3\times S^3$ using cohomogeneity one methods. Finally, Cortés and Vásquez considered finite quotients of the homogeneous examples in \cite{CV15}.

The term nearly parallel $G_2$ is used here to refer to any Riemannian manifold whose metric cone has holonomy contained in $\Spin(7)$. In particular, due to the natural inclusions $\Sp(2) \subseteq \SU(4) \subseteq \Spin(7)$, all Sasaki-Einstein and 3-Sasaki manifolds in dimension $7$ are also nearly parallel $G_2$-manifolds. Sasakian geometry has been studied extensively (cf. \cite{BG08}, \cite{Sparks11}), and a great number of examples are known. When the cone of a nearly parallel $G_2$-manifold has holonomy exactly $\Spin(7)$, the structure is called a proper nearly parallel $ G_2$-structure. Two distinct families of examples are known. First, on every 3-Sasaki manifold, the second Einstein metric is a proper nearly parallel $G_2$ metric, as proved in \cite{FKMS97}. This metric is obtained by squashing the $\SU(2)$-fibres of a 3-Sasaki manifold, referred to as squashed nearly parallel $G_2$-manifolds. Secondly, the homogeneous examples were classified in \cite{FKMS97}:
    \begin{itemize}
        \item $(S^7, g_{squashed}) \cong \Sp(2) \times \Sp(1)/ \Sp(1) \times \Sp(1)$,
        \item $N(k,l)\cong \SU(3)/S^1_{k,l}$ for $k,l \in \Z$,  where $S^1 \rightarrow \SU(3)$ as $z \mapsto (z^k, z^l, z^{-(k+l)})$, and
        \item $\SO(5)/\SO(3)$, where the isotropy representation is the unique seven-dimensional irreducible representation $\SO(3) \rightarrow G_2 \subseteq \SO(7)$.
    \end{itemize}
Notice that $(S^7, g_{squashed})$ and $N(1,1)$ also belong to the family of squashed $G_2$ metrics. Podestà \cite{Pode21} and Singhal \cite{singhal23} have made recent progress in constructing cohomogeneity one proper nearly parallel $G_2$-manifolds.

In the early 2000s, Hitchin \cite{Hitchin00} \cite{Hitchin01} introduced a variational approach to nearly Kähler and nearly parallel $G_2$-structures. He defined volume functionals $\mathcal{L}$ (resp. $\mathcal{P}$) over the space of stable forms on $M^6$ (resp. $M^7$) and showed that critical points were precisely nearly Kähler (resp. nearly parallel $G_2$) structures. In this notes, we continue the study of these functionals and show that their gradient flow for a suitable indefinite inner product corresponds to special holonomy metrics on the product space $\R_+ \times M$ , as detailed in Propositions \ref{SU3coneflow} and \ref{G2coneflow}. We describe the spectral decomposition of the Hessian of these two functionals, recovering the results for infinitesimal deformations of Moroianu, Nagy and Semmelmann \cite{MNS08}, and Alexandrov and Semmelmann \cite{AS12}.

We introduce two new Hitchin-type functionals, $\mathcal{Q}$ and $\mathcal{T}$ that generalise the original functionals $\mathcal{L}$ and $\mathcal{P}$. In the case of nearly Kähler structures, the domain of the functional will be 
$$ \mathcal{U}=  \left\{ \omega \in \Omega^2 \st d \omega \textrm{~is~stable,~} \omega \textrm{~is~stable~and~positive,~} \omega^2 \textrm{~is~exact} \right\}\;.$$
In contrast with Hitchin's original framework, where the metric conditions are guaranteed by the Euler--Lagrange equations, every point in $\mathcal{U}$ carries a natural $\SU(3)$-structure (cf. Prop. \ref{extraordianrysu3}), which motivates the geometric interest in the new Hitchin functional. Moreover, we have 
\begin{theorem*} The new Hitchin functional $\mathcal{Q}: \mathcal{U}\rightarrow \R$ satisfies the following.
    \begin{enumerate}[label=(\roman*)]
        \item The Einstein--Hilbert action is a lower bound for $\mathcal{Q}$. The two only coincide along rescalings of nearly Kähler structures (cf. Prop. \ref{compareNK_EH} ).
        \item Critical points have a well-defined index with respect to a natural indefinite inner product(cf. Thm. \ref{specHess6}), which provides a lower bound for the Einstein co-index (cf. Prop. \ref{index_bound_NK}).
        \item The gradient flow for this inner product is not parabolic, even after a DeTurck trick.
        \item There is an explicit connection between the spectrum of $\delta^2 \mathcal{L}$ and $\delta^2 \mathcal{Q}$ (cf. Prop. \ref{index2index}).
    \end{enumerate}
\end{theorem*}
In the nearly parallel $G_2$ case, we consider the domain $\mathcal{V}=\{ \psi \in \Omega^4 \st \psi \textrm{~is~stable~and~exact} \}$. We have 
\newline
\begin{theorem*}The new Hitchin functional $\mathcal{T}: \mathcal{V} \rightarrow \R $ satisfies the following.
    \begin{enumerate}[label=(\roman*)]
        \item The Einstein--Hilbert action is a lower bound for $\mathcal{T}$. The two only coincide along rescalings of nearly parallel $G_2$ structures (cf. Prop. \ref{compareG2_EH}).
        \item Critical points have a well-defined index with respect to a natural indefinite inner product, which provides a lower bound for the Einstein co-index (cf. Prop. \ref{index_bound_GS}).
        \item The gradient flow with respect to this inner product is not parabolic.
        \item There is an explicit connection between the spectrum of $\delta^2 \mathcal{P}$ and $\delta^2 \mathcal{T}$ (cf. Prop. \ref{2nd21st}).
    \end{enumerate}
\end{theorem*}
Since the gradient flows of these new functionals are well-posed initial value problems, one could naively hope to flow to a critical point under suitable starting conditions. However, in Proposition \ref{nonparabolicity}, we prove that the nearly Kähler flow is not strictly parabolic, even after using DeTurck's trick. In particular, we can not guarantee the existence of short-term solutions for these flows. 

The Hitchin index for nearly Kähler structures was explored by Karigiannis and Lotay in their study of the moduli space of \( G_2 \)-conifolds in \cite{LK20}. In their analysis, the Hitchin index naturally emerged through analytical techniques, contributing to the virtual dimension of the moduli space. This suggests, together with our results, that the Hitchin index for \( G_2 \)-holonomy conifolds is analogous to the stability index defined by Joyce \cite{Joy032} for special Lagrangians and by Lotay \cite{Lot07} for coassociative submanifolds. A more detailed discussion of this analogy is provided in the outlook. We conjecture that a similar result will hold for \( \Spin(7) \)-holonomy structures, where only the asymptotically conical case has been explored by Lehmann \cite{Leh21}.

Computing the index implies understanding the spectrum of a second-order partial differential operator. In the homogeneous case, one can reduce the computation of the Hitchin index to an algebraic problem, using the Peter-Weyl theorem and Frobenius reciprocity. Karigiannis and Lotay \cite{LK20} (cf. \cite{MS10}) use this to compute the Hitchin index for the four homogeneous nearly Kähler structures. In combination with the work of Schwahn \cite{Sch22}, we have
\begin{theorem*}
    The homogeneous nearly Kähler structures are Hitchin stable. Their Einstein index is the sum of the second and third Betti numbers.
\end{theorem*}
In the sequel of this work \cite{ESF24c}, we study the index problem in the cohomogeneity one examples of Foscolo and Haskins \cite{FH17}.  

\section*{Acknowledgments}
The author is indebted to his PhD supervisors, Simon Donaldson and Lorenzo Foscolo, for their continuous support, guidance and
encouragement, throughout the completion of this project. We also thank Uwe Semmelmann for his comments on an early version of the manuscript.

This work was supported by the Engineering and Physical Sciences Research Council [EP/S021590/1]. The EPSRC Centre for Doctoral Training in Geometry and Number Theory (The London School of Geometry and Number Theory), University College London.

\section{Stable forms and Hitchin functionals}\label{Hitchin}
In the early 2000s', Nigel Hitchin \cite{Hitchin00} \cite{Hitchin01} showed how some geometric structures can be realised as critical points of suitable functionals over a class of suitably generic forms called stable. Hitchin classified all the possible cases in his original papers.
\begin{definition}
    Let $V^n$ be a n-dimensional real vector space. A form $w\in \Lambda^p(V^*)$ is stable if the orbit of $w$ under the induced $\GL(V)$ action is open.
\end{definition}
If the stabiliser of a stable form is a subgroup of $\SL(V)$, there is an invariant volume form associated with each stable form. Similarly, if the stabiliser is compact, there is an invariant inner product associated with the stable form.

We assume that $\operatorname{Stab}(w) \subseteq \SL(V)$ for the remainder of the section. Assigning a stable form its invariant volume form defines a $\GL(V)$-invariant map $\vol: \Lambda_+^p(V^*) \rightarrow \Lambda_+^n(V^*)\cong \R^*$. In particular, taking $\mu \in \R$, this map satisfies
$$\vol(\mu^p w)= \mu^n \vol(w)\;.$$
In other words, $\vol$ is homogeneous of degree $n/p$. Its derivative defines an invariant element $\widehat{w} \in (\Lambda^p V^*)^*\otimes \Lambda^nV^* \cong \Lambda^{n-p}V^*$,  
 \begin{equation} \label{volumederivative}
     \frac{\delta}{ \delta  \alpha } \vol(w) =  \alpha  \wedge \widehat{w} \;.
 \end{equation} 
 We call $\widehat{w}$ the Hitchin dual of $w$. Using Euler's formula for $\alpha =w$, we obtain the relation 
\begin{equation} \label{eulerfactor}
    w \wedge \widehat{w}= \frac{n}{p} \vol(w)\;.
\end{equation}
 This discussion extends naturally to the setting of smooth manifolds. For convenience, we will reduce our discussion to the set-up where $M$ is a smooth closed oriented manifold. We say a smooth $p$-form $\rho\in \Omega^p(M)$ is stable, if the restriction of $\rho$ at a point $x\in M$, $\rho_x$, is stable for every point.

The existence of stable smooth forms is only obstructed by the reduction of the frame bundle to the corresponding stabiliser. The set of stable forms will be denoted by $\Omega_+^p(M)$, which will be open in $\Omega^p(M)$ whenever it is non-empty. 
The map $\vol$ above extends to a smooth map $\vol: \Omega^p_+(M) \rightarrow \Omega^n(M)$ and we can define the corresponding volume functional $V$ by integrating against the fundamental class of the manifold, known as the Hitchin functional. Since stable forms form an open set, one can study the variational properties of the functional $V$. The main result, due to Hitchin, is an application of Stokes' theorem.
\begin{theorem}[\cite{Hitchin00}]
    A closed stable form $\rho\in \Omega^p_+(M)$ is a critical point of $V$ within its cohomology class if and only if its Hitchin dual is closed; i.e. $d\widehat{\rho}=0$.
\end{theorem}
Let us look at two concrete instances of the Hitchin functional, described in detail in \cite{Hitchin00}.
\begin{example}[6-dimensions \cite{Hitchin00}]
If we are in the case $n=6$, $p=3$, we have a locally decomposable complex volume form $\rho+i\widehat{\rho}$. The critical point condition implies that the complex volume form is closed and that the almost complex structure it induces is integrable. 
\end{example}
\begin{example}[7-dimensions \cite{Hitchin00}]
Let $\varphi$ be a stable 3-form on a 7-manifold. Then $\varphi$ induces a metric on $M$ by $g_\varphi(X,Y)=\frac{1}{6}(X \lrcorner \varphi)\wedge (Y \lrcorner \varphi) \wedge \varphi$. Critical points are holonomy $G_2$-manifolds. Moreover, they are all local maxima. 
\end{example}
These examples illustrate the interest in studying these Hitchin functionals. As a further motivation, one can look at the gradient flow of the volume functional in the case of $G_2$-structures along a fixed cohomology class. The corresponding flow is called the $G_2$ Laplacian flow. Short-time existence and uniqueness of this flow were proved by Bryant and Xu \cite{BX11}. It remains a central object of study in special holonomy.

In the examples above, there is the critical assumption that the cohomology class over which we are trying to optimise is non-trivial. Otherwise,
critical points cannot exist by standard Hodge theory.

If one wants to restrict to stable exact forms, one needs to impose a further non-degeneracy condition, in the form of a Lagrange multiplier. In dimensions 6 and 7, Hitchin \cite{Hitchin01} obtained two new functionals for the exact stable case and showed that the critical points of these functional are nearly Kähler and nearly parallel $G_2$-structures, respectively. These notes study these functionals and their variations further, comparing them to two new examples of Hitchin-like functionals. 
\section{Nearly Kähler structures}
We begin by recalling some well-known results on $\SU(3)$-structures and nearly Kähler manifolds. These results are classic and have been collected for convenience. 
\begin{definition}
    An $SU(3)$-structure on a manifold $M^6$ is a reduction of its frame bundle to an $\SU(3)$-
    principal bundle. A manifold equipped with a choice of frame reduction is called an $\SU(3)$-manifold.
\end{definition}
Equivalently, $M^6$ is equipped with a pair of stable differential forms $(\omega, \rho)\in \Omega^2(M)\times\Omega^3(M)$, with $\operatorname{Stab}(\rho)= \SL(3, \C)$, satisfying the following algebraic constraints:
\begin{equation} \omega \wedge \rho=0 ~~~~~~~~~~~~~~~~~~~ \frac{1}{3!} \omega^3 = \frac{1}{4} \rho \wedge  \widehat{\rho} \;, \label{algebraic_constraints}
\end{equation}
with $\omega$ positive with respect to the complex structure induced by $\rho$. The algebraic constraints \eqref{algebraic_constraints} guarantee that the stabiliser of a stable pair is precisely $\SU(3)=\Sp(6,\R)\cap\Sl(3, \C)$. Similarly, one could have chosen a pair $(\rho, \sigma)\in \Omega^3 \times \Omega^4$ with  $\widehat{\sigma}=\omega$ and $\sigma= \frac{\omega^2}{2}$, satisfying the above conditions. 

The inclusion $\SU(3) \subset \SO(6)$ implies that $M$ inherits a metric $g$ form the $\SU(3)$-structure. Indeed, let $J$ be the almost complex structure induced by $\rho$. Then the condition $\omega \wedge \rho=0$ is equivalent to $\omega$ is of type $(1,1)$ with respect to $J$. Then $g \coloneqq \omega(\cdot, J\cdot)$ is the required metric, and its induced volume form coincides with $\frac{1}{3!}\omega^3$. Moreover, since $\SU(3)\subset\SU(4)\cong \Spin(6)$ is the stabiliser of any $v\in \C^4\setminus\{0\}$, an $\SU(3)$-structure is equivalent to the choice of a metric and spin structure on $M$, together with a nowhere vanishing spinor. In particular, the only obstruction for a manifold to admit an $\SU(3)$-structure is for $M$ to be orientable and spinnable, i.e. $w_1(M)=w_2(M)=0$.

Using the metric, we will identify  $T^*M$ and $TM$. In particular, for $X$ a vector field, when we write $JX$ and treat it as a 1–form we mean $g(JX, \cdot)$. To avoid confusion arising from the fact that $g(JX, \cdot)=-Jg(X, \cdot)$, we will instead treat $X$ as a $1$-form throughout and instead distinguish between $df$ and $\nabla f$, as in \cite{Fos17}.

The reduction of the structure group of $M$ to $\SU(3)$ leads to a decomposition of $\Lambda^* T^*M$ into irreducible $\SU(3)$-representations. This decomposition
is well-known and most commonly phrased in terms of the $(p,q)$-decomposition of the complexification $\Lambda^* T^*M \otimes \C$, induced by the almost complex structure $J$. However, for our purposes, it is more convenient to use real irreducible representations.
\begin{lemma} \label{decompositionSU(3)}
Let $(M,J,\omega, \rho)$ be an $SU(3)$-manifold. We have the following decomposition into irreducible $\SU(3)$-representations:
$$\Lambda^2 T^*M = \Lambda^2_1 \oplus \Lambda^2_6 \oplus \Lambda^2_8\;,$$
with $\Lambda^2_1= \langle \omega \rangle$, $\Lambda^2_6= \{X \lrcorner \rho \mid X \in TM \}$ and $\Lambda^2_8$ is the space of primitive $(1,1)$ forms. Similarly, we get
$$\Lambda^3 T^*M = \Lambda^3_{1\oplus1} \oplus \Lambda^3_6 \oplus \Lambda^3_{12}\;,$$
where $\Lambda^3_{1\oplus1}= \operatorname{Span} \{\rho , \widehat{\rho}\}$, $\Lambda^3_6= \{X \wedge \omega \mid X \in T^*M \}$ and $\Lambda^3_{12}$  is the space of primitive $(1,2)+(2,1)$ forms. Using the Hodge star, we get the decomposition for $\Lambda^4 T^*M$. Finally, we have
$$ Sym^2(TM)\cong \R \oplus Sym^2_+ \oplus Sym^2_- \cong \Lambda^0 \oplus \Lambda^2_8 \oplus \Lambda^3_{12}$$
where $Sym^2_+= \{ S \in Sym^2(TM) \mid JS=S \;, \;\operatorname{tr}(S)=0 \}$ and $Sym^2_-= \{ S \in Sym^2(TM) \mid JS=-S\}$, and the last isomorphism is an isomorphism of representations.
\end{lemma}
The isomorphisms  $I: \Sym^2_+ \rightarrow \Lambda^2_8$ and $\Upsilon : \Sym^2_- \rightarrow \Lambda^3_{12}$ are given by  
$I(S)= S_*(\omega)$ and $\Upsilon(S)= S_*(\rho)$, where an endomorphism $S\in \operatorname{End}(\R^6)$ acts on a $k$-form $\alpha $ by 
\begin{equation} \label{endaction}
    S_*(\alpha)(X_1, \dots, X_k) \coloneqq - \sum_{i=1}^k \alpha(X_1, \dots, SX_i, \dots , X_k) \;.
\end{equation}
These decompositions carry over to the spaces of smooth sections of each bundle. We denote $\Omega^k_m:=\Gamma(\Lambda^k_m)$.

We can identify $\Lambda^2_6$ with $\Lambda^1$. The adjoint operator to the contraction $X \mapsto X \lrcorner \rho$, denoted by $c: \Lambda^2_6 \rightarrow \Lambda^1$, is given explicitly by $c(\beta)=-*(\beta\wedge \widehat{\rho})$. This allows us to introduce an auxiliary differential operator, which can be understood as a generalisation of the curl for $\SU(3)$-structures.
\begin{definition}
    Let $(M,\omega, \rho)$ be an $\SU(3)$-manifold. We define the curl operator
\begin{align*}
    \curl:\Omega^1 & \rightarrow \Omega^1\\
    X &\mapsto c(dX)=-*(dX\wedge \widehat{\rho})\;.
\end{align*}
\end{definition}
More generally, the Hodge star operator is compatible with the $\SU(3)$-structure and therefore induces an isomorphism between the representations appearing in different degrees. 

Given an $\SU(3)$-structure, one can study its intrinsic torsion. The torsion can be decomposed into irreducible representation and identified with the components of $d\omega$, $d\rho$ and $d\widehat{\rho}$, as discussed by Gray and Hervella in \cite{GH80} (cf. \cite{CS02}):
\begin{proposition}\label{torsion6}
    Let $(\omega, \rho)$ be an $SU(3)$-structure. Then there exists forms $\tau_0, \widehat{\tau}_0 \in \Omega^0$, 
    $\tau_1, \widehat{\tau}_1\in \Omega^1$, $\tau_2, \widehat{\tau}_2 \in \Omega^2_8$ and $\tau_3\in \Omega^3_{12}$ such that 
    \begin{align*}
        d\omega &= 3 \tau_0 \rho + 3 \widehat{\tau}_0 \widehat{\rho}+ \tau_1\wedge \omega +\tau_3 \;,\\
        d\rho &= 2 \widehat{\tau}_0 \omega^2 + \widehat{\tau}_1 \wedge \rho + \tau_2 \wedge \omega \;,\\
        d\widehat{\rho}&=-2 \tau_0 \omega^2 - J\widehat{\tau}_1 \wedge \widehat{\rho} + \widehat{\tau}_2\wedge \omega \;.
    \end{align*}
These fully describe the torsion of the $\SU(3)$-structure.
\end{proposition}
We are interested in the induced $G_2$-structure on the metric cone. In particular, we can consider the following two classes of $\SU(3)$-structures, characterised in terms of the Gray-Hervella torsion decomposition. If the $G_2$-structure on the cone is closed, i.e. $d\phi=0$, we say $M$ carries a closed $\SU(3)$-structure. The torsion is concentrated in $\tau_0=1$ and $\widehat{\tau_2}$. These structures were originally studied by physicists in the context of string theory, under the name of LT-structures, introduced in \cite{LT05}. If the $G_2$-structure on the cone is both closed and coclosed, $M$ carries a nearly Kähler structure and the only non-vanishing torsion term is $\tau_0=1$.

Since torsion-free $\SU(3)$-structures (i.e. Calabi-Yau 3-folds) are Ricci flat, the Bianchi identities imply that the Ricci tensor of an $\SU(3)$-structure can be fully described in terms of its torsion, as exploited by Bedulli and Vezzoni in \cite{BV07}. In particular, they get the following expression for the scalar curvature. 
\begin{lemma}[{\cite[Thm 3.4]{BV07}}]\label{scalar_curv_NK}
    Let $(M, \omega, \rho)$ be an $\SU(3)$-structure. The scalar curvature of the associated metric is given by
    $$s_g= 30 (\tau_0^2 +\widehat{\tau_0}^2)+  2d^*(\tau_1+ \widehat{\tau_1}) - \abs{\tau_1}^2  + 4 \langle \tau_1, \widehat{\tau_1}\rangle- \frac{1}{2} \left(\abs{\tau_2}^2+\abs{\widehat{\tau_2}}^2 + \abs{\tau_3}^2 \right)\;,$$
    where $\tau_i$ and $\widehat{\tau_i}$ are the torsion forms of Proposition \ref{torsion6}.
\end{lemma}
We have an explicit formula for the linearisation of Hitchin's duality map from Section \ref{Hitchin} in terms of irreducible representations. We collect the result here as it is useful for computations in the next section:
\begin{proposition}[\cite{Hitchin00} Section 3.3] \label{linear6}
Given $(\rho, \sigma)$defining an $\SU(3)$-structure, consider $\chi= \chi_1+\chi_6+\chi_8 \in \Omega^4$ and $\gamma=\gamma_{1\oplus1}+\gamma_6+\gamma_{12} \in \Omega^3$. Then 
\begin{enumerate}[label=(\roman*)]
    \item The derivative of the Hitchin dual map at $\sigma$ in the direction of $\chi$ is
    $$\frac{d}{dt} \widehat{(\sigma+ t \chi)} \Bigr|_{t=0} \coloneqq \mathcal{K}(\chi)= \frac{1}{2}*\chi_1 + *\chi_6 - *\chi_8\;.$$
    \item The derivative of the Hitchin dual map at $\rho$ in the direction of $\gamma$ is
    $$\frac{d}{dt} \widehat{(\rho+ t \gamma)} \Bigr|_{t=0} \coloneqq \mathcal{I}(\gamma)= *\gamma_{1\oplus1} + *\gamma_6 - *\gamma_{12}\;.$$
\end{enumerate}
\end{proposition}
As a straightforward corollary, we get
\begin{lemma}\label{hitchintrick6}
Let $\mathcal{I}: \Omega^3 \rightarrow \Omega^3$ and $\mathcal{K}: \Omega^4 \rightarrow \Omega^2$ be the maps defined in Proposition \ref{linear6}. For any $X \in \Omega^1$, we have $\lie_X \widehat{\rho}=  \mathcal{I} \lie_X \rho$ and $\lie_X \widehat{\sigma}=  \mathcal{K} \lie_X \sigma$.
\end{lemma}
Finally, nearly Kähler manifolds enjoy an adapted Hodge decomposition akin to complex manifolds, which will be key in studying the second variations of the Hitchin functionals. Such decomposition is obtained by studying a Dirac-type operator and its mapping properties. The main decomposition result is due to Foscolo \cite{Fos17}.
\begin{theorem}[\cite{Fos17} Proposition 3.22]\label{hodgedecom6}
    Let $(M^6, \omega, \rho)$ be a nearly Kähler manifold that is not isometric to the round $6$-sphere, and denote by $\mathrm{K}$ the set of Killing fields of $(M^6,g)$. Then, the following holds.
    \begin{enumerate}[label=(\roman*)]
        \item $\Omega^3 = \{X \wedge \omega \st X \in \mathrm{K}\} \oplus d\Omega^2_{1\oplus 6} \oplus d^*\Omega^4_{1} \oplus \Omega^3_{12}$. In particular, there is an $L^2$-orthogonal decomposition  $\Omega^3_{exact} = d\Omega^2_{1\oplus 6} \oplus \Omega^3_{12, exact}$.
        \item For every $\chi\in \Omega^4$, there exists unique $X\in \mathrm{K}$, $Y\in \mathrm{K}^\perp$, $f\in \Omega^0$ and $\chi_0\in \Omega^4_8$ such that
        $$\chi = (X \wedge \widehat{\rho}) + d(JY \wedge \omega + f\widehat{\rho}) + \chi _0\;, $$
        where $\mathrm{K}^\perp$ is the space $L^2$ complement to the subspace of Killing fields. In particular, there is an $L^2$-orthogonal decomposition $\Omega^4_{exact} = d\Omega^3_{1\oplus 6} \oplus \Omega^4_{8, exact}$.
    \end{enumerate}
\end{theorem}
Furthermore, we have the following type characterisation that will also be useful: 
\begin{proposition}[{\cite[Prop. 3.6 \& Lemma 3.7]{Fos17}}]\label{useful6-2}
    Let $(M, \omega, \rho)$ be a nearly Kähler manifold.  For $\beta \in \Omega^2_{8,coclosed}$, $d\beta \in \Omega^3_{12}$ and for $\gamma\in \Omega^3_{12,coclosed}$, $d \gamma \in \Omega^4_8$.
\end{proposition} 

\subsection{The nearly Kähler Hitchin functional}
Let us assume that $M^6$ is a closed spinnable manifold, so it admits an $\SU(3)$-structure. In $6$-dimensions, we have a non-degenerate pairing between $\Omega^3_{exact}$ and $\Omega^4_{exact}$, defined as follows:
\begin{align*}
    P: \Omega^3_{exact} \times \Omega^4_{exact} & \rightarrow \R \\
        (\gamma, \chi) & \mapsto  \int_M \beta \wedge \chi =   - \int_M \gamma \wedge \xi  \;,
\end{align*}
where $d\beta=\gamma$ and $d\xi=\chi$. This pairing follows from the Stokes' 
 theorem and the identification $\Omega^3 / \Omega^3_{closed}\cong \Omega^4_{exact}$.  With it, one can construct an indefinite inner product on $\Omega^3_{exact} \times \Omega^4_{exact}$:
$\{(\gamma_1, \chi_1),(\gamma_2, \chi_2)\}\coloneqq P(\gamma_1, \chi_2) + P(\gamma_2, \chi_1)\;.$

Let $\mathcal{R} \subseteq \Omega^3_{exact} \times \Omega^4_{exact}$ be the space of stable exact forms $(\rho, \sigma)$, with $\omega= \hat{\sigma}$ positive with respect to $\rho$. In \cite{Hitchin01}  introduced the functional that plays the analogue role for nearly Kähler structures as the examples described in the introduction:
\begin{equation} \label{hitchin_functional_6}
    \begin{aligned}
    \mathcal{L}: \mathcal{R} & \rightarrow \R \\
        (\rho, \sigma) & \mapsto  3\int_M \vol_\rho+ 4\int_M \vol_\sigma-12 P(\rho,\sigma)\;.
\end{aligned}
\end{equation}
 He showed that its critical points are nearly Kähler structures. Indeed, let $\delta\rho= \gamma= d\eta \in \Omega^3_{exact}$ and $\delta\sigma= \chi= d\xi \in \Omega^4_{exact}$. Then $\delta \vol_\rho = \gamma \wedge \widehat{\rho}$ and $\delta \vol_\sigma = \chi \wedge \widehat{\sigma}= \chi \wedge \omega$, so 
 \begin{align*}
     \delta \mathcal{L}&= -3 \int_M  \big(\widehat{\rho} +4 \beta \big) \wedge \gamma  +4 \int_M \big( \omega -3\alpha) \wedge \chi= -3 \int_M  \big(d\widehat{\rho} +4 \sigma \big) \wedge \eta  - 4 \int_M \big( d\omega -3\rho) \wedge \xi \;,
 \end{align*}
 where $d\alpha =\rho$ and $d\beta=\sigma$. Thus, the Euler--Lagrange equations are
 \begin{equation} \label{hitchineq}
     d\widehat{\rho}= -4 \sigma ~~~~~~~~~~~ d\omega = 3 \rho\;.
 \end{equation} 
\begin{proposition}[Thm. 6 \cite{Hitchin01}]
    The critical points of $\mathcal{L}$ are nearly Kähler structures.
\end{proposition} 
\begin{proof}
We need to check that equations \eqref{hitchineq} imply that $(\omega, \rho)$ satisfy the $\SU(3)$ conditions. Indeed, we have
$$\omega \wedge \rho = \frac{1}{3} \omega \wedge d\omega = \frac{1}{3} d\sigma = \frac{-1}{12} d^2 \widehat{\rho}=0\;,$$
so $\omega$ is of type $(1,1)$ with respect of the complex structure defined by $\rho$. Similarly, we have
$$\frac{\omega^3}{3!}= \frac{1}{3}\omega \wedge \sigma = \frac{-1}{12}\omega \wedge d\widehat{\rho}= \frac{-1}{12}\Big(d(\omega \wedge \widehat{\rho}) - d\omega \wedge \widehat{\rho}\Big) = \frac{1}{4} \rho \wedge \widehat{\rho}\;. \hfill \qedhere$$
\end{proof}
We find it convenient to work with the gradient flow of $\mathcal{L}$ with respect to the pairing $\{\cdot, \cdot\}$ rescaled it by $(1/3, 1/4)$:
\begin{equation} \label{NKconeflow}
\frac{\partial \sigma}{\partial t}= -\left(d\widehat{\rho}+ 4 \sigma \right) ~~~~~~~~~~~ \frac{\partial \rho}{\partial t}= \left(d\omega - 3 \rho\right)\;.
\end{equation}
We have no good argument for this rescaling, beyond that it has some desirable properties and allows us to motivate the study of this functional. Notice that a global rescaling can be obtained by suitably rescaling $\mathcal{L}$ (or the inner product $\{\cdot, \cdot\}$. However, the relevance of the rescaling is that is different on 3-forms and 4-forms. We have
\begin{proposition}
The rescaled gradient flow preserves the $\SU(3)$-condition.
\end{proposition}
\begin{proof}
Since $0=d\omega^2/2= \omega\wedge d\omega$, it follows that $\tau_1=\pi_6(d\omega)=\pi_6(d\sigma)=0$. Similarly, since $\rho$ is exact, we have $\widehat{\tau_1} =\pi_6(d\widehat{\rho})= J\pi_6(d\rho)=0$. Thus
$$ \frac{\partial}{\partial t}(\omega\wedge \rho)= \frac{\partial}{\partial t} \omega \wedge \rho + \omega \wedge \frac{\partial}{\partial t} \rho= - * \big(\pi_6(d\widehat{\rho})\big)\wedge \rho+\pi_6(d\omega) \wedge \omega=0\;, $$
proving the condition $\omega \wedge \rho=0$ is preserved. Now, by Equation \eqref{eulerfactor}, we have $\vol_\rho = \frac{1}{2} \rho \wedge \widehat{\rho}$ and $\vol_\sigma =  \frac{2}{3} \sigma  \wedge \omega = \frac{1}{3} \omega^3$. Thus, it suffices to check that
$\vol_\rho = \vol_\sigma$  is preserved under the flow. By Equation \eqref{volumederivative}, we have 
$$
    \frac{\partial}{\partial t} \vol_\rho = \frac{\partial \rho}{\partial t} \wedge \widehat{\rho}= 
   (d\omega -3 \rho) \wedge \widehat{\rho} = -d \widehat{\rho} \wedge \omega - 3\rho \wedge \widehat{\rho}= \frac{\partial \sigma}{\partial t}\wedge \omega + 4 \sigma \wedge \omega - 3\rho \wedge \widehat{\rho} =
     \frac{\partial}{\partial t} \vol_\sigma + 6 \left( \vol_\sigma - \vol_\rho \right)\;. \qedhere$$ 
\end{proof}
The main result that motivates our study of the Hitchin functions is its relation with metric cones with special holonomy. Explicitly, we have
\begin{proposition}\label{SU3coneflow}
 Let $\big(\rho(t), \sigma(t)\big)$, $t\in(a,b)$, be a family of exact stable forms on $M^6$ defining an $\SU(3)$-structure, with associated metric $g(t)$. Then the metric $\overline{g}=dr^2+r^2g(log(r))$ in $(e^a,e^b) \times M$ has holonomy inside $G_2$ if and only if the $\SU(3)$-structures satisfy the rescaled gradient evolution equations.
 
\end{proposition}
\begin{proof}
Given an $\SU(3)$-structure on $\Sigma$, we get a $G_2$-structure on the cone $C(\Sigma) $ by setting $\phi = dr\wedge r^2 \omega + r^3\rho $ and $\psi=*\phi= -dr\wedge r^3\widehat{\rho} + r^4\sigma\;.$  The condition $\operatorname{Hol}(g_\phi)\subseteq G_2$ is equivalent to the 3-form $\phi$ being closed and coclosed. Thus, by differentiating, we get
\begin{eqnarray*}
    0=d\phi= - dr \wedge r^2 d_\Sigma \omega + 3 r^2 dr \wedge \rho + r^3 dr \wedge \frac{\partial \rho}{\partial r} ~~~~ \implies & ~~~~~ r \frac{\partial \rho}{\partial r}= d_\Sigma \omega - 3 \rho  \\
    0=d\psi= dr \wedge r^3 d_\Sigma \widehat{\rho} + 4 r^3 dr \wedge \sigma + r^4 dr \wedge \frac{\partial \sigma}{\partial r} ~~~~ \implies& ~~~~~ r \frac{\partial \sigma}{\partial r}= - d_\Sigma \widehat{\rho} - 4 \sigma\;.
\end{eqnarray*}
where $d_\Sigma$ is just the restriction of the exterior differential $d$ along $\Lambda ^* T^*\Sigma$ and we used that $d_\Sigma \rho= d_\Sigma \sigma=0$. This condition is precisely the rescaled gradient flow equations under the change of variables $r=e^t$. The converse follows. 
\end{proof}
\begin{remark}
Theorem 8 in \cite{Hitchin01} is very similar to the above propositions. The method is essentially the same, but Hitchin applies it to a different functional and considers unweighted metrics, $\overline{g}=dt^2+ g_\Sigma(t)$. It is worth comparing the two. We can replace our Lagrange multiplier form $12$ to $12\lambda$ and consider the metric cone with angle $2\pi \lambda$, with $G_2$-structure given by $\phi =\frac{dr}{\lambda} r^2 \omega + r^3 \rho$ and metric $g_\lambda= \Big(\frac{dr}{\lambda}\Big)^2 +r^2g_\sigma$. The condition that the cone has holonomy in $G_2$ is then equivalent to the rescaled gradient flow that now depends on $\lambda$. In this case, the required relation between $r$ and $t$ becomes $r=e^{\lambda t}$. The resulting metric is conformal to the metric $dt^2 + g_\Sigma$ by a factor of $e^{2\lambda t}$. Thus, after suitable rescaling, the limiting metric $\lambda\rightarrow 0$ recovers Hitchin's result.
\end{remark}
\begin{remark}
In his proof, Hitchin considers a Hamiltonian flow induced by the symplectic pairing induced by $P$ rather than the gradient flow approach. With the Hamiltonian approach, one can see the vanishing condition  $\omega \wedge \rho=0$ as the vanishing of the moment map induced by the $\operatorname{Diff}(M)$ action. This approach would have worked equally well in our setup.
\end{remark}
We now focus on the second variation of $\mathcal{L}$:
\begin{proposition}
Let $(\gamma_1,\chi_1),(\gamma_2,\chi_2) \in \Omega^3_{exact} \times \Omega^4_{exact}$, with $\gamma_i=d\eta_i$ and $\chi_i=d\xi_i$ for $i=1,2$. The second variation of $\mathcal{L}$ is given by
$$\delta^2 \mathcal{L}=\int_M -3(d\mathcal{I}\gamma_2 +4\chi_2)\wedge \eta_1  -4 (d\mathcal{K}\chi_2 -3 \gamma_2)\wedge \xi_1\;.$$
In particular the Hessian of $\mathcal{L}$ at a critical point with respect to the pairing $\{ \cdot, \cdot\} $ is
$$H^\mathcal{L}(\gamma, \chi) = \Big(4d\mathcal{K}\chi -12 \gamma, -3d\mathcal{I}\gamma -12\chi\Big)\;. $$
\end{proposition}
\begin{proof}
By Proposition $\ref{linear6}$, if $\delta\rho=\gamma$, then $\delta \widehat{\rho}=\mathcal{I}\gamma$. Similarly, if $\delta\sigma=\chi$, then $\delta \omega=\mathcal{K}\omega$. Combining this with our formula for the first variation, we get the desired formula. The computation of the Hessian with respect to the pairing $\{\cdot, \cdot\}$ is now immediate.
\end{proof}
We want to study the spectral properties of $H^\mathcal{L}$. More concretely, the equations
\begin{equation} \label{hess6}
\begin{array}{r@{}l}
    -3d\mathcal{I} \gamma &=(\mu +12 )\chi\\
    4d\mathcal{K} \chi &=(\mu+12 )\gamma\;, 
\end{array}
\end{equation}
for $\gamma \in \Omega^3_{exact}$ and $\sigma\in \Omega^4_{exact}$.
Since the functional $\mathcal{L}$ is invariant under $\operatorname{Diff}(M)$, it is convenient to work on a slice to the orbit of the diffeomorphism group. We use the same strategy as \cite{Fos17}.

Let $(\omega, \rho)$ be a nearly Kähler structure not isometric to the round $S^6$ and $\mathcal{O}$ be the orbit of $\operatorname{Diff}_0(M)$ in $\Omega^3_{exact} \times \Omega^4_{exact}$ going through $(\omega, \rho)$. The tangent space to this orbit is spanned by $(\lie_X \rho, \lie_X \sigma)$, for $X\in \mathrm{K}^\perp \subseteq \Omega^1$, where $\mathrm{K}$ is the set of Killing fields and the complement is taken with respect to the $L^2$ metric. 
Using the Hodge decomposition of Theorem \ref{hodgedecom6}, we can parametrise $(\gamma, \chi)\in \Omega^3_{exact} \times \Omega^4_{exact}$ explicitly by 
$$\gamma = \lie_X \rho + d(f\omega) +\gamma_0 ~~~~~~~~~~ \chi = \lie_Y \sigma + d(g\widehat{\rho}) +\chi_0\;;$$
with $f,g \in \Omega^0$, $X,Y \in \mathrm{K}^\perp$, $\gamma_0 \in \Omega^3_{12,exact}$  and $\chi_0 \in \Omega^4_{8,exact}$. In particular, it follows that taking $X=0$ or $Y=0$ defines a complement to the tangent space of the diffeomorphism action. Let 
$$\mathcal{W}=\big\{\big(d(f\omega) +\gamma_0, \lie_Y \sigma + d(g\widehat{\rho}) +\chi_0\big)\big\} \subseteq \Omega^3_{exact} \times \Omega^4_{exact}\;,$$
for $f, g \in \Omega^0$ , $Y\in \mathcal{K}^\perp$, $\gamma_0 \in \Omega^3_{12, exact}$ and $\chi_0\in \Omega^4_{8, exact}$. Taking the appropriate Hölder norm completions, we get
\begin{proposition}[\cite{Nor08} Theorem 3.1.4 and 3.1.7]
There exists a slice to the diffeomorphism group action in $\Omega^3_{exact} \times \Omega^4_{exact}$, whose tangent space is given by $\mathcal{W}$.
\end{proposition}
We can now study the spectral properties of the second variation of the functional $\mathcal{L}$. We have
\begin{proposition}\label{hess6-1}
Assume $(M^6, \omega, \rho)$ is not isometric to the round $6$-sphere. Under the Hodge decomposition, the equations \eqref{hess6} are equivalent to
\begin{subequations}
\begin{align} 
    -8g&= (\mu+12)f\;, \label{1st}\\
    -9f &= (\mu+12)g\;, \label{2nd}\\
    Y+\frac{1}{3}dg &= \frac{\mu+12}{12}X\;, \label{3rd}\\
    X-\frac{1}{4}df &= \frac{\mu+12}{12}Y\;,   \label{4th}\\
    d*\gamma_0&= \frac{(\mu+12)}{3}\chi_0\;, \label{5th}\\
    d*\chi_0 &= -\frac{(\mu+12)}{4}\gamma_0\;. \label{6th}
\end{align}
\end{subequations}
\end{proposition}
\begin{proof}
As above, let 
\begin{align*}
    \gamma &= \lie_X \rho + d(f\omega) +\gamma_0= \lie_X \rho + df \wedge \omega+ 3f \rho +\gamma_0\;,\\
    \chi &=\lie_Y \sigma + d(g\widehat{\rho}) +\chi_0 = \lie_Y \sigma + dg\wedge \widehat{\rho} - 4g \sigma  +\chi_0\;;
\end{align*}
with $f,g \in \Omega^0$, $X,Y \in \mathrm{K}^\perp$, $\gamma_0 \in \Omega^3_{12,exact}$  and $\chi_0 \in \Omega^4_{8,exact}$, in virtue of Theorem \ref{hodgedecom6}. By the definition of $\mathcal{I}$ and $\mathcal{K}$, and Lemma \ref{hitchintrick6}, we get
$$\mathcal{I}\gamma = \lie_X \widehat{\rho} + Jdf \wedge \omega +3f\wedge \widehat{\rho} - *\gamma_0 ~~~~~~~~~~ 
\mathcal{K}\chi = \lie_Y \omega + dg \lrcorner \rho - 2g\omega -*\chi_0\;.$$
Now, since $(\omega, \rho)$ is nearly Kähler, we get 
\begin{align*}
    d\mathcal{I}\gamma &= -4 \lie_X \sigma + d(Jdf \wedge \omega) +d(3f\wedge \widehat{\rho}) -d*\gamma_0 = -4 \lie_{X-\frac{1}{4}df} \sigma+ d(3f\wedge \widehat{\rho}) -d*\gamma_0\;,\\
    d\mathcal{K}\chi &= 3\lie_Y \rho +  d(dg \lrcorner \rho) - d(2g\omega) -d(*\chi_0) =  3\lie_{Y + \frac{1}{3}dg} \rho - d(2g\omega) -d(*\chi_0) \;.
\end{align*}
Plugging this back in \eqref{hess6} and since the Hodge decomposition is orthogonal, the system \eqref{1st}-\eqref{6th} follows.
\end{proof}
\begin{proposition} \label{indexaslaplaceeigenvalueNK}
The eigenforms of $H^\mathcal{L}$ are constant functions and solutions to 
\begin{equation} \label{lap6}
    \Delta \gamma =  \frac{(\mu+12)^2}{12} \gamma
\end{equation}
for $\gamma\in \Omega^3_{12, exact}$. In particular, the spectrum of $\mathcal{H}^\mathcal{L}$ is discrete and has finite multiplicity for each $\mu$.
\end{proposition}
\begin{proof}
First, equations \eqref{1st} and \eqref{2nd} imply $72fg=(\mu+12)^2fg$. The only solution to this equation with $fg\neq0$ corresponds to $\mu= -12 \pm 6\sqrt{2}$. If we further impose the gauge fixing condition $X=0$, equations \eqref{3rd} and \eqref{4th} become
\begin{eqnarray*}
    Y+\frac{1}{3}dg =0 ~~~~~~~~~~~~~~~
    3df \pm 6\sqrt{2} Y = \pm \frac{\sqrt{2}}{2}\left(Y - \frac{1}{3}dg\right)=0\;,
\end{eqnarray*}
since $f= \mp \frac{2\sqrt{2}}{3} g$ by equation \eqref{2nd}. Thus, $Y=df=dg=0$, so $f$ and $g= \pm \tfrac{3\sqrt{2}}{4} f$ must be constant, with associated eigenvalue $\mu=12 \pm 6\sqrt{2}$. We have reduced our spectral problem to the PDE system \eqref{5th}-\eqref{6th}
\begin{equation} \label{sys61}
    d*\gamma_0= \frac{(\mu+12)}{3}\chi_0 ~~~~~~~~~~~~~~~~~~~~~~~~~~    d*\chi_0 = -\frac{(\mu+12)}{4}\gamma_0\;, 
\end{equation} 
with $(\gamma_0, \chi_0) \in \Omega^3_{12,exact} \times \Omega^4_{8,exact}$. If $\mu=-12$, $(\gamma_0, \chi_0)$ are harmonic exact forms and thus zero. Thus, we may assume $ \mu\neq 12$. In this case, this PDE system is equivalent to \eqref{lap6}. If $\gamma_0$ satisfies \eqref{sys61}, then 
$$ \Delta \gamma_0 =  d d^* \gamma_0= - \frac{(\mu+12)}{3} d* \chi_0 = \frac{(\mu+12)^2}{12} \gamma_0\;.$$
Conversely, if $\gamma_0$ satisfies \eqref{lap6} and $\mu\neq -12$, the pair $\big(\gamma_0, \frac{3}{(\mu+12)}d*\gamma_0\big)$ satisfies \eqref{sys61}:
$$d*\chi_0= \frac{3}{(\mu+12)} d*d*\gamma_0 = \frac{-3}{(\mu+12)} \Delta \gamma_0 = -\frac{(\mu+12)}{4}\gamma_0\;. \qedhere$$
\end{proof}
\begin{remark}
    The case $\mu=0$ corresponds to the nullity of $\mathcal{H}^\mathcal{L}$, i.e. infinitesimal deformations of the nearly Kähler structure. As expected, we recover the result of \cite{MNS08} and \cite{Fos17} on infinitesimal deformations of nearly Kähler structures.
\end{remark}
\subsection{The closed Hitchin functional} \label{sectionclosedSU3}
The Euler--Lagrange equations associated with the functional $\mathcal{L}$ resemble the first variation of a Hamiltonian functional. This similarity suggests a natural approach: to seek out and analyse Lagrangians that correspond to the Hitchin functional $\mathcal{L}$ when interpreted in a Hamiltonian framework. In particular, we will treat the exact $3$-form $\rho$ as the moment variable within this setting. To formalise this approach, consider the map
\begin{align*}
Cl :\Omega^2  & \rightarrow \Omega^3 \times \Omega^4\\
\omega &\mapsto \Big(\frac{1}{3}d\omega, \frac{1}{2}\omega^2\Big)\;,
\end{align*}
and let $\mathcal{U}:= Cl^{-1}(\mathcal{R})$ be the preimage of exact stable forms $(\rho, \sigma)$. The pullback of the Hitchin functional $\mathcal{L}$ under $Cl$ will be the corresponding Lagrangian functional. 

The space $\mathcal{U}$ is a priori quite mysterious. In particular, important questions to answer are under which conditions the space is non-empty, whether it is path-connected or simply connected. The following key result shows that $\mathcal{U}$ has a very natural geometric description:
\begin{proposition}\label{extraordianrysu3}
There is a one-to-one map between $\mathcal{U}$ and the 
 set of $\SU(3)$-structure with torsion supported in the classes $\tau_0= e^f$, $\widehat{\tau}_1$ and $\widehat{\tau}_2$.
\end{proposition}
This connection with $\SU(3)$-structures justifies the choice of $\rho$ as the moment variable, rather than $\sigma$, for which a result like Proposition \ref{extraordianrysu3} is not available.
\begin{proof}
Let $\omega\in \mathcal{U}$. Then the 3-form $\tilde{\rho}\coloneqq \frac{1}{3}d\omega$ is stable and satisfies $\omega \wedge \tilde{\rho}= \frac{1}{6}d\omega^2=0$ since $\omega\in \mathcal{U}$. Thus, the pair $(\omega, \tilde{\rho})$ defines an $\SU(3)$-structure modulo the volume compatibility condition.  Now, let $u=e^f\in \mathcal{C}^\infty(M)$ be the unique function such that
$$ \frac{\omega^3}{3!}= \frac{1}{4u^2} \tilde{\rho} \wedge \widehat{\tilde{\rho}} = \frac{1}{4}\frac{\tilde{\rho}}{u}  \wedge \widehat{\left(\frac{\tilde{\rho}}{u} \right)} \;.$$
Then the pair $(\omega,\rho)=(\omega, \frac{1}{u} \tilde{\rho})=(\omega, \frac{1}{3u} d\omega)$ defines an $\SU(3)$-structure. It follows easily that the torsion of this $\SU(3)$-structure is given by $\tau_0=u=e^f$, $\widehat{\tau}_1=-df$ and $\widehat{\tau}_2$, with $d^*\widehat{\tau}_2= Jdf$. 

Conversely, given an $\SU(3)$-structure $(\omega, \rho)$ with these torsion classes, it is clear that $d\omega =3\tau_0 \rho$ is stable, provided $\tau_0$ is everywhere non-zero.
\end{proof}
In particular, closed $\SU(3)$-structures are a closed subset of $\mathcal{U}$, obtained by enforcing $f=0$. Thus, we could think of $\mathcal{U}$ as conformally closed $\SU(3)$-structures.

Let us study the pullback of the Hitchin functional under $Cl$. We denote this pullback by $\mathcal{Q}$. We have 
\begin{equation} \label{new_hitchin_functional_6}
        \mathcal{Q}= Cl^*\mathcal{L}= 3 \int_M vol_{1/3 d\omega}+ 8\int_M vol_{\omega} - 12P\Big(\frac{1}{3}d\omega, \frac{\omega^2}{2}\Big)= \frac{1}{3}\int_M \vol_{d\omega} - 4 \int_M\vol_\omega\;,
\end{equation}
where used the fact that $\vol_\sigma = 2\vol_\omega$ as a straightforward application of \eqref{eulerfactor}. Similarly, we can pull back the inner product. Let $[\cdot, \cdot]= \frac{1}{2}Cl^*\{\cdot, \cdot\}$. For $\alpha, \beta \in T\mathcal{U}$, we have $$[\alpha, \beta ]=  \frac{1}{3}\int_M \alpha \wedge \beta \wedge \omega =  \frac{1}{3}\int_M \alpha \wedge \mathcal{K}^{-1}(\beta)\;,$$ 
where $\mathcal{K}$ is the linearisation of the Hitchin dual map from Proposition \ref{linear6} with respect to the $\SU(3)$-structure from Proposition \ref{extraordianrysu3}. This follows from noticing that, for any $4$-form $\chi$, the $2$-form $\mathcal{K}(\chi)$ is the unique form that satisfies $\mathcal{K}(\chi) \wedge \omega= \chi$. The following result further motivates the interest in the functional $\mathcal{Q}$.
\begin{proposition}\label{compareNK_EH}
Consider the  map $F: \mathcal{U}\rightarrow \operatorname{Met}(M)$ that maps the $\SU(3)$-structure to its underlying metric, and let  $\hat{\mathcal{S}}= F^*\mathcal{S}$ the pullback of the Einstein--Hilbert action \eqref{EH-def}. The Hitchin functional $\mathcal{Q}$ is bounded below by $\widehat{\mathcal{S}}$. Moreover, the two functionals coincide if and only if the $\SU(3)$-structure is a constant multiple of a nearly Kähler structure.
\end{proposition}
\begin{proof}
Using the formula for the scalar curvature in Lemma \ref{scalar_curv_NK}, the pulled-back Einstein--Hilbert action \eqref{EH-def} can be written as: 
$$\widehat{\mathcal{S}}(\omega)=  \frac{1}{5}\int_M  s_g -20 \vol_g = \frac{1}{5}\int_M \left(30 \tau_0^2  -\frac{1}{2}\abs{\widehat{\tau_2}}^2\right)- 20\vol_g  = \int_M 6 \tau_0^2 -4-\frac{1}{10}\abs{\widehat{\tau_2}}^2 \vol_g \;.$$

Similarly for $\mathcal{Q}$, we have $d\omega= 3 \tau_0  \rho$, and so, $\vol_{d\omega} = (3 \tau_0)^2 \vol_\rho$. Substituting in the definition of $\mathcal{Q}$:
$$ \mathcal{Q}= \frac{1}{3} \int_M \vol_{d\omega}- 4 \int_ M \vol_{\omega} = \int_M 6 \tau_0^2 -4  \vol_g \;,$$
where we used $\vol_g= \frac{1}{2}\vol_\rho= \vol_\omega$ . The claim follows. 
\end{proof}
Let us study the variational properties of the Lagrangian functional $\mathcal{Q}$. The first variation of $\mathcal{Q}$ along $\beta$ is 
\begin{align*}
    \delta \mathcal{Q}&= \int_M \frac{1}{3} d\beta \wedge \widehat{d\omega} - 4 \frac{\omega^2}{2} \wedge \beta =-\frac{1}{3}\int_{M}\Big( d(\widehat{d\omega}) +12 \frac{\omega^2}{2}\Big)\wedge \beta\;.
\end{align*}
The gradient flow of $\mathcal{Q}$ with respect to $[\cdot, \cdot]$ is $\partial_t \omega = -\mathcal{K}(d(\widehat{d\omega})-6 \omega\;.$ This flow becomes slightly more enlightening if we consider the induced flow for $\sigma=\frac{\omega^2}{2}$:
\begin{equation} \label{lapflow}
    \frac{\partial \sigma}{\partial t} = - d(\widehat{d*\sigma}) -12 \sigma = dd^*\sigma -12 \sigma = \Delta \sigma - 12 \sigma \;,
\end{equation}
since $\widehat{d\omega}= *d\omega$. We call this flow the nearly Kähler Laplacian flow. 
\begin{proposition}
    The critical points of $\mathcal{Q}$ are nearly Kähler structures.
\end{proposition}
\begin{proof}
With respect to the induced $\SU(3)$-structure, the fixed points of the gradient flow are 
$$0=\Delta \sigma -12 \sigma = -3 d (\tau_0 \widehat{\rho}) -12 \sigma = 12\tau_0^2\sigma - 3\tau_0 \widehat{\tau_2}\wedge \omega -12 \sigma\;, $$ which implies $\tau_0=1$ and $\widehat{\tau_2}=\widehat{\tau_1}=0$, as needed.
\end{proof}
The second variation of $\mathcal{Q}$ at a nearly Kähler structure is given by 
\begin{equation} \label{2ndvariationNK2}
    \frac{ \partial ^2 \mathcal{Q}}{\partial \alpha \partial \beta} = \frac{1}{3}\int_M d\alpha \wedge \mathcal{I}d\beta  - 12 \omega \wedge \alpha\wedge\beta = \frac{-1}{3}\int_M \alpha \wedge (d\mathcal{I}d\beta  + 12 \omega \wedge \beta) \;.
\end{equation}
We can associate a symmetric endomorphism $\mathcal{H}^\mathcal{Q}$ to the second variation via the pairing $[\cdot, \cdot]$, which we refer to as the Hessian of $\mathcal{Q}$. Before studying the spectral properties of $\mathcal{H}^\mathcal{Q}$, it is convenient to get a more manageable description of $T_\omega\mathcal{U}$.
\begin{proposition} \label{kiso}
There is an isomorphism $T_\omega\mathcal{U} \cong \mathcal{K} \big( \Omega^4_{exact} \big)$
\end{proposition}
\begin{proof}
Recall that $\mathcal{U}= Cl^{-1}(\mathcal{R}) = \big\{ \omega \in \Omega^2 \st \omega~\mathrm{is~stable}, ~ d\omega ~\mathrm{is~stable},\omega^2 ~\mathrm{is~exact} \big\}$. The stability and positivity conditions are open, so we only need to study the constraint of $\omega^2$ being exact. Its linearisation along $\delta \omega = \alpha$ is given by $ \delta \omega^2 =2 \omega \wedge \alpha= 2\mathcal{K}^{-1}(\alpha)$. 
\end{proof}
Since $\mathcal{K}$ is a pointwise linear isomorphism, we will instead study the spectral properties of $\overline{\mathcal{H}}^\mathcal{Q}\coloneqq \mathcal{K}^{-1} \circ \mathcal{H}^\mathcal{Q} \circ \mathcal{K}: \Omega^4_{exact} \rightarrow \Omega^4_{exact}$. Explicitly, we want to solve the equation 
\begin{equation} \label{spec7}
    d\mathcal{I} d \mathcal{K} \chi =  -(\mu+12) \chi
\end{equation} 
for $\mu \in \R$ and $ \chi \in \Omega^4_{exact}$. Since the functional $\mathcal{Q}$ is invariant under the action of the diffeomorphism group, it is convenient to work on a slice to the orbit of the diffeomorphism group. Let $\omega \in \mathcal{U}$ be a nearly Kähler structure and $\mathcal{O}$ be the orbit of $\operatorname{Diff}_0(M)$ in $T\mathcal{U}$ going through $\omega$. The tangent space to this orbit is spanned by $\lie_X \omega$, for $X\in \mathrm{K}^\perp \subseteq \Omega^1$, where $\mathrm{K}$ is the set of Killing fields and the complement is taken with respect to the $L^2$ metric. Under the isomorphism of Proposition \ref{kiso}, the image of the tangent space of orbit is spanned by $\mathcal{K}^{-1}\lie_X \omega = \lie_X \sigma$, for $X\in \mathrm{K}^\perp \subseteq \Omega^1$, where we used Lemma \ref{hitchintrick6}. Now, by Theorem \ref{hodgedecom6}, we can parametrise $\chi \in \Omega^4_{exact}$ by $ \chi = \lie_X \sigma + d(f\widehat{\rho}) + \chi_0\;,$ where $f\in \Omega^0$, $X \in \mathrm{K}^\perp$ and $\chi_0\in \Omega^4_{8, exact}$. In particular, taking $X=0$, we get that $\mathcal{W}=\big\{ d(f\widehat{\rho}) + \chi_0 \big\} \subseteq \Omega^4_{exact}$ is a complement to the tangent space of the diffeomorphism orbit. Arguing as before and taking the appropriate Hölder norm completions, we can integrate $\mathcal{W}$ into a gauge slice, and we can prove
\begin{theorem}\label{specHess6}
Assume $(M^6, \omega, \rho)$ is not isometric to the round 6-sphere. Under the Hodge decomposition and gauge fixing, Eq. \eqref{spec7} is equivalent to
\begin{subequations}
    \begin{align} \label{sys6}
        (\mu+6)f&=0\;,\\
        df&=0\;, \\
        \Delta \chi_0&= (\mu+12) \chi_0\;;
    \end{align}
\end{subequations}
 where $f\in \Omega^0$ and $\chi_0\in \Omega^4_{8, exact}$. Solutions are $f=C$ with $C\in \R$ for $\mu=-6$ and the solutions to 
 $$\Delta \chi_0 = (\mu+12) \chi_0 ~~~~~~~ d\chi_0 = 0\;.$$
 In particular, the spectrum is discrete and has finite multiplicities. 
\end{theorem}
\begin{proof}
Let $$ \chi = \lie_X \sigma + d(f\widehat{\rho}) + \chi_0\;,$$
with $f\in \Omega^0$, $X \in \mathrm{K}^\perp$ and $\chi_0\in \Omega^4_{8, closed}$. As in the proof of Proposition \ref{hess6-1}, we have 
 $$ d\mathcal{K}\chi = 3\lie_X \rho +  d(df \lrcorner \rho) - d(2f\omega) -d(*\chi_0) =  3\lie_{X + \frac{1}{3}df} \rho - 2df \wedge \omega - 6f\wedge \rho -d(*\chi_0) \;.$$
 Since $\chi_0$ is closed, $d(*\chi_0) \in \Omega^3_{12}$, by Proposition \ref{useful6-2}. Thus, acting by $\mathcal{I}$, we get
 $$ \mathcal{I}d\mathcal{K}\chi=  3\lie_{X + \frac{1}{3}df} \widehat{\rho} - 2Jdf \wedge \omega - 6f\widehat{\rho} +*d(*\chi_0) \;,$$
  and finally, acting by $d$ once more, we get 
 $$ d\mathcal{I}d\mathcal{K}\chi=  -12\lie_{X + \frac{1}{3}df} \sigma  - d(2Jdf \wedge \omega) - 6d(f\widehat{\rho}) + d*d(*\chi_0) = -12\lie_{X + \frac{1}{2}df} \sigma - 6d(f\widehat{\rho}) -\Delta \chi_0\;.$$
 Plugging this back in, and since the Hodge decomposition is unique, the system follows.
\end{proof}
As before, the case $\mu=0$ recovers the infinitesimal deformations of the $\SU(3)$-structure. Moreover, the following is a straightforward corollary of Proposition \ref{indexaslaplaceeigenvalueNK}: 
\begin{proposition}\label{index2index}There is a two-to-one correspondence between the eigenforms of $\mathcal{H}^\mathcal{L}$ and $\mathcal{H}^\mathcal{Q}$.
\end{proposition}
This result motivates the following definition:
\begin{definition}\label{NKindex}
    Let $(M^6,\omega, \rho)$ a nearly Kähler manifold. We define the Hitchin index of the nearly Kähler structure $\operatorname{Ind}_{(\omega, \rho)}$ as the number of negative eigenvalues of the Hessian endomorphism $\mathcal{H}^\mathcal{Q}: \mathcal{W}\rightarrow \mathcal{W}$ at $(\omega, \rho)$ minus one.
\end{definition}
Equivalently, let $\mathcal{E}(\lambda)=  \left\{ \beta \in \Omega^2_8 \; \Bigr| ~d^* \beta =0, \; \Delta \beta= \lambda \beta \right\}$. Then the Hitchin index is 
\begin{equation} \label{NKindexeq}
 \operatorname{Ind}_{(\omega, \rho)}=  \sum_{\lambda \in (0,12)} \dim~ \mathcal{E}(\lambda)\;. 
\end{equation}
Notice that a prior this definition of the index is different from the usual Morse definition as the index of the quadratic form $ \delta ^2 \mathcal{Q}$, since the pairing used to define the endomorphism $\mathcal{H}^\mathcal{Q}$ is indefinite. However, we will show that the two quantities are connected in this case. We prove 
\begin{proposition} \label{index_vs_coindex}
The Morse co-index of the second variation $\delta^2 \mathcal{Q}$ restricted to $\Omega^4_{8, exact}$ equals the Hitchin index. 
\end{proposition}
Let us evaluate $\delta^2\mathcal{Q}(\beta, \beta)$ for $\beta\in T_\omega \mathcal{U}$. Equation \eqref{2ndvariationNK2} and Proposition \ref{kiso} yield
\begin{equation}
    \delta^2 \mathcal{Q} =- \frac{1}{3} \int_M \mathcal{K}( \chi) \wedge \left( d \mathcal{I} d \mathcal{K}(\chi) +12 \chi \right)\;,
\end{equation}
for $\chi\in \Omega^4_{exact}$. Fixing the diffeomorphism slice, we can take $\chi= d(f\widehat{\rho}) + \chi_0$ for $f \in \Omega^0$ and $\chi \in \Omega^4_{8,closed}$. By the computations of the proof of Proposition \ref{hess6-1}, we have
\begin{align*}
   \delta^2 \mathcal{Q} & = -\frac{1}{3}\int_M \langle \Delta \chi_0 -12 \chi_0, \chi_0\rangle - 48 \int_M f^2 \vol_g + 2 \int_M (df \lrcorner \rho) \wedge\left[d (Jdf \wedge \omega) -  df \wedge \widehat{ \rho} \right] \\
   &= -\frac{1}{3}\int_M \langle \Delta \chi_0 -12 \chi_0, \chi_0\rangle + 8 \int_M\langle \Delta f -6f, f \rangle \;.
\end{align*}
Thus, the second variation has two distinct behaviours on the two subspaces of $\mathcal{W}= \{d(f\hat{\rho}) \}  \oplus \Omega^4_{8, exact}$, similar to the Einstein--Hilbert case (cf. Theorem \ref{Einstein2ndvar}). We refer to the first subspace as the conformal deformations since they give rise to conformal changes of the metric.

An interesting first result is the Hitchin stability of the homogeneous examples.
\begin{theorem} \label{index_homogeneous_NK}
    Let $(M^6, \omega, \rho)$ be one of the four homogeneous nearly Kähler manifolds. Then 
    $$\operatorname{Ind}_{(\omega,\rho)} =0 \;.$$
\end{theorem}
The main tool we need is a version of the Peter-Weyl for naturally reductive homogeneous spaces and a comparison between the Hodge Laplacian and the canonical Laplacian. On a nearly Kähler structure, besides the Levi-Civita connection, there exists another metric connection, called the canonical connection, with the property that $\operatorname{Hol}\left(\nabla^{can}\right) \subseteq \SU(3)$. The relationship between these two connections is given explicitly by 
\begin{equation}
    \nabla ^{can} = \nabla ^{LC} -  \frac{1}{2}\widehat{\rho} \;. \label{canonicalconnection}
\end{equation}
The canonical Laplacian is the connection Laplacian associated with this connection, $\Delta^{can} = \left( \nabla^{can}\right)^* \nabla^{can}$. 

Both results mentioned above are due to Moroianu and Semmelmann. They were collected in \cite{MS10} and \cite{MS11} in their investigation of infinitesimal nearly Kähler and Einstein deformations.
\begin{lemma} \cite[Prop. 4.5]{MS11}  \label{MS1}
    Let $(M^6, \omega , \rho)$ be a nearly Kähler manifold, ${\Delta}^{can}$ the induced connection Laplacian. 
    For $\beta \in \Omega^2_8$, we have the Weitzenböck-type formula:
    $$ (\Delta -  {\Delta}^{can}) \beta = (J d^*\beta) \lrcorner \rho\;.$$
    In particular, both Laplacians coincide on coclosed forms of type $\Omega^2_8$.
\end{lemma}
\begin{proposition} \cite[Lemmas 5.2 \& 5.4]{MS10}\label{MS2}
    Let $(G/H, \omega, \rho)$ be a naturally reductive nearly Kähler manifold and consider $\rho: H \rightarrow \aut(E)$ a representation of $H$ and $EM=G \times_\rho E$ the induced vector bundle. Then, the Peter-Weyl formalism and Frobenius reciprocity imply
    \begin{equation*}
        L^2(EM) = \overline{\bigoplus_{\gamma \in \operatorname{Irr}(G)} V_\gamma \times \operatorname{Hom}_H(V_\gamma, E)}\;,
    \end{equation*}
    where $\operatorname{Irr}(G)$ denotes the set of irreducible $G$-representations. Under this decomposition, the canonical Laplacian is given by ${\Delta}^{can}=-12 \operatorname{Cas}^G_\gamma$, where $\operatorname{Cas}^G_\gamma$ is the Casimir of the representation $V_\gamma$, computed with respect to the Killing form.
\end{proposition}
\begin{proof}[Proof of Theorem \ref{index_homogeneous_NK}]
Using the computations of Moroianu and Semmelmann in \cite{MS10}, Karigiannis and Lotay \cite[Prop. 6.3]{LK20} showed that the homogeneous nearly Kähler structures in $\C P^3 \cong $, $S^3\times S^3\cong \SU(2)^3/\Delta \SU(2)$ and the flag $F_{1,2}\cong \SU(3)/\operatorname{T}^2)$ are stable. Thus, only the case of the round sphere $S^6\cong G_2/\SU(3)$ remains. We start by computing the Casimir operator of $G_2$. Let $\omega_1$ and $\omega_2$ be the short and long fundamental weights respectively, so $V_{1,0}$ is the fundamental 7-dimensional $G_2$-representation and $V_{0,1}$ is its adjoint representation.

Since $G$ is simple, the Freudenthal formula (cf. \cite{MS10}) allows us to compute the value of the Casimir operator on a representation of highest weight $\gamma$. We have $\operatorname{Cas}^G_\gamma- \langle \gamma , \gamma +2 \rho\rangle_B$, where $\rho$ is the half-sum of positive roots and $\langle \cdot, \cdot\rangle_B$ is the Killing form. In the case of $G_2$-structure, $\rho= \omega_1 +\omega_2$, and so, for an irreducible representation of highest
weight $(\lambda, \mu)$, its Casimir operator is given by 
\begin{align*}
    \operatorname{Cas}_{G_2}(\lambda, \mu) &= -\lambda(\lambda+2) \norm{\omega_1}^2_B + \mu(\mu+2)\norm{\omega_2}^2_B + 2(\lambda \mu + \lambda + \mu) \langle \omega_1, \omega_2 \rangle\\
     &=  -\frac{1}{12} \left( \lambda(\lambda+2) + 3\mu(\mu+2) + 3(\lambda \mu + \lambda + \mu) \right) 
     \;.
\end{align*}
    Therefore, in virtue of Lemma \ref{MS1} and Proposition \ref{MS2}, the Hodge Laplacian on coclosed forms of type $\Omega^2_8$ is given by
    $$ \Delta \beta =  \sum_{(\lambda, \mu) \in \operatorname{Irr}(G)} \left( \lambda(\lambda+2) + 3\mu(\mu+2) + 3(\lambda \mu + \lambda + \mu) \right) \pi_\gamma(\beta)\;. $$
    In particular, the only highest weight for which the eigenvalue of the canonical Laplacian is smaller than $12$ is $(1,0)$, the fundamental $7-$ dimensional representation.
    The space of primitive $(1,1)$-forms can be identified with the adjoint representation of $\SU(3)$. By dimensional reasons it is clear that $\operatorname{Hom}_{\SU(3)} \left(V_{1,0},  \mathfrak{su}(3) \right)=0$, and so $(S^6, g_{round})$ is stable.
\end{proof}
\begin{remark}
    Notice that a prior, this computation is only valid if one defines the Hitchin index using Eq. \eqref{NKindexeq} since the gauge slice is not valid in the round sphere case. However, one can retrace the proof of Theorem \ref{hodgedecom6} and show that the discussion can be adapted without significant changes. We omit the details.
\end{remark}
The natural next question is the study of the Hitchin functionals and the index problem for the remaining two known examples of nearly Kähler structures. Some results for these examples were proved by the author \cite{ESF24c}.

\section{Nearly parallel $G_2$-structure}
We start by recalling well-known results on $G_2$-structures and nearly parallel $G_2$-manifolds. 
\begin{definition}
    A $G_2$-structure on a manifold $M^7$ is a reduction of its frame bundle to a $G_2$-principal bundle. A manifold equipped with such a frame reduction is called a $G_2$-manifold \footnote{The term $G_2$-manifold in the literature is sometimes reserved for manifolds carrying a metric with holonomy contained in $G_2$}.
\end{definition}
Equivalently, $M^7$ is equipped with a smooth, stable differential form $\varphi \in \Omega^3_+(M)$, in the sense of Hitchin. Similarly, we could have fixed an orientation choice and chosen a stable 4-form $\psi$, where $\phi$ and $\psi$ will be Hitchin duals to each other. Notice that $\phi\in \Omega^3_+$ determines an orientation, whilst $\psi \in \Omega^4_+$ does not.

Moreover, since $G_2\subset \Spin(7)$ is the stabiliser of any $v\in \R^7 \cong \operatorname{Im}(\mathbb{O})$, a $G_2$-structure is equivalent to the choice of a spin structure together with a nowhere vanishing spinor. In particular, the only obstruction for $M^7$ to admit a $G_2$-structure is for $M$ to be orientable and spinnable. Finally, the inclusion $G_2\subseteq \SO(7)$ implies that $M$ inherits a metric $g$ from the $G_2$-structure. Indeed, for $X, Y$ vector fields, we can define $g_\phi: \operatorname{Sym}^2(TM) \rightarrow \R$ as
    $$g_\phi(X,Y) \vol_\phi= \frac{1}{6} (X \lrcorner \phi) \wedge (Y \lrcorner \phi) \wedge \phi\;.$$
It is worth noting that sometimes, the opposite orientation convention is chosen. We follow the same convention as Bryant \cite{Bry05}, Joyce \cite{Joy00}, Salamon and Walpuski \cite{SW10} and Dwivedi and Singhal \cite{SS20}. Bryant-Salamon \cite{BS89}, Harvey-Lawson \cite{HL82} and Karigiannis and Lotay\cite{LK20} follow the opposite convention. Using the associated metric, we will henceforth identify $TM$ and $T^*M$ and identify vector fields and $1$-forms without further distinction.
\subsubsection*{Form decomposition and torsion}
The reduction of the structure group to $G_2$ leads to a decomposition of $\Lambda^*(T^*M)$ into irreducible representations of $G_2$. 
\begin{lemma}\label{formdecom7}
Let $(M^7, \varphi)$ be a $G_2$-manifold. The spaces $\Lambda^0T^*M\cong \R$ and $\Lambda^1 T^*M \cong \R^7$ are irreducible with respect to the induced action of $G_2$.
The spaces $\Lambda^2$ and $\Lambda^3$ decompose orthogonally as 
$$\Lambda^2= \Lambda^2_7 \oplus \Lambda^2_{14} ~~~~~~~~~~~~~~~~~~~~~~ \Lambda^3= \Lambda^3_1 \oplus \Lambda^3_7 \oplus \Lambda^3_{27}\;,$$
where $\Lambda^k_l$ has pointwise dimension $l$. They are described by 
\begin{align*}
    \Lambda^2_7= &\{X \lrcorner \varphi \st X\in \Gamma(TM)\}= \{\beta\in \Lambda^2 \st \star (\varphi \wedge \beta)=2\beta\}\\
    \Lambda^2_{14}=& \{\beta \in \Lambda^2 \st \beta \wedge \psi=0\}= \{\beta\in \Lambda^2 \st \star (\varphi \wedge \beta)=-\beta\}\\
\Lambda^3_1=  \Lambda^0 \cong \langle \varphi \rangle ~~~~~ & 
    \Lambda^3_7= \{X \lrcorner \psi \st X\in \Gamma(TM)\} ~~~~~
    \Lambda^3_{27}= \{\gamma \in \Lambda^3 \st \gamma \wedge \psi=0, \gamma\wedge \varphi=0 \}
\end{align*}
The decomposition for $\Lambda^k$ for $k>3$ follows from the fact that the metric is $G_2$-invariant and $\Lambda^k= \star \Lambda^{7-k}$. Similarly, we have
$$ \Sym^2(T^*M)\cong \langle g \rangle \oplus  Sym^2_0(TM) \cong \Lambda^0 \oplus \Lambda^3_{27}\;,$$
where $\Sym^2_0(T^*M)$ denotes the traceless symmetric tensors and the last isomorphism is an isomorphism of representations, $i_\phi: \Sym^2_0 \rightarrow \Lambda^3_{27}$, given explicitly by $i_\phi (S) = S_*(\phi)$, as in Eq. \eqref{endaction}.
\end{lemma}
These decompositions carry over to the spaces of smooth sections of each of the bundles. We denote $\Omega^k_m = \Gamma(\Lambda^k_m)$. Using the metric, we canonically identify the space of 1-forms with smooth vector fields. Using the metric we can identify $\Lambda^2_7$ and $\Lambda^3_7$ with $\Lambda^1$ using the maps $X\mapsto X \lrcorner \phi$ and $X\mapsto X \lrcorner \psi$. In particular, we can define the generalised curl operator, originally introduced in \cite{Kar08}.
\begin{definition}
Let $(M^7, \phi)$ be a $G_2$-structure. We define the curl of a $1$-form as
\begin{align*}
    \curl:\Omega^1 & \rightarrow \Omega^1\\
    X &\mapsto *(dX\wedge \psi)\;.
\end{align*}
\end{definition}
Similarly, one obtains useful identities from the different incarnations of each representation. A comprehensive list can be found in \cite[Lemma 4.37]{SW10}.

Following the discussion on \cite{CS02}, we have the following decomposition of the torsion in the irreducible representation of $d\phi$ and $d\psi$.
\begin{proposition}\label{torsion7}
    Let $(\phi, \psi)$ be an $G_2$-structure. Then there exists forms $\tau_0\in \Omega^0$, $\tau_1 \in \Omega^1$, $\tau_2 \in \Omega^2_{14}$ and $\tau_3\in \Omega^3_{27}$ such that 
 $$ d\phi = 4\tau_0 \psi + 3\tau_1 \wedge \phi + *\tau_3 ~~~~~~~~~~~~ d\psi = 4\tau_1 \wedge \psi + \tau_2\wedge \phi $$
\end{proposition}
\begin{proposition}
Let $C(M)$ be a metric cone whose holonomy is contained in $\Spin(7)$. Then, the $G_2$-structure on the link $M$ has torsion concentrated in $\tau_0=1$. 
\end{proposition}
As in the case of $\SU(3)$-structures, since holonomy $\Spin(7)$-metrics are Ricci flat, we can use the Bianchi identities to describe the Ricci and scalar curvature in terms of the torsion. 
\begin{lemma}[{\cite[Eq. (4.28)]{Bry05}}] \label{scalar_curv_G2}
    Let $(M, \phi)$ be a $G_2$-structure. The scalar curvature of the associated metric is given by
    $$s_g= 42\tau_0^2 + 12 d^*\tau_1 + 30 \abs{\tau_1}- \frac{1}{2}\abs{\tau_2}^2 - \frac{1}{2}\abs{\tau_3}^2\;.$$
\end{lemma}
Finally, we have an explicit formula for the linearisation of Hitchin's duality map in terms of irreducible representations. 
\begin{proposition}[{\cite[Lemma 20]{Hitchin00}, \cite[Sect. 6]{Bry05}}] \label{linear7}
Given $\psi\in \Omega^4(M)$ defining a $G_2$-structure, consider $\chi= \chi_1+\chi_7+\chi_{27} \in \Omega^4$ of small $C^0$-norm so that $\psi+\chi$ is still a stable $4$-form. Then the image $\widehat{\chi}$ of $\chi$ under the linearisation of Hitchin’s duality map at $\psi$ is
    $$ \frac{d}{dt}\left(\widehat{ \psi + t \chi}\right)\coloneqq \mathcal{J}(\chi)= \frac{3}{4}*\chi_1 + *\chi_7 - *\chi_{27}\;.$$
Similarly, the metric $g_\phi$ changes by 
$$\delta g_\phi =  \frac{1}{2} \chi_1 + \frac{1}{2}i_\phi^{-1}\chi_{27}\;,$$
where $i_\phi^{-1}: \Omega^3_{27} \rightarrow \Gamma( Sym^2_0(TM))$ is the inverse of the smooth extension of the map defined in Lemma \ref{formdecom7}.
\end{proposition}
The following lemma is a useful consequence of this result in combination with the Lie derivative. 
\begin{lemma}\label{hitchintrick}
Let $\J(\chi)=*\big(\frac{3}{4}\pi_1+\pi_7-\pi_{27}\big)(\chi)$ from $\Omega^4$ to $\Omega^3$ defined above. For any $X \in \Omega^1$, we have 
\begin{subequations}
    \begin{align}
            \lie_X \phi &=  \J \lie_X \psi \\
            \lie_X g &=  \frac{1}{2} \pi_1 (\lie_X \phi) g+  \frac{1}{2}i^{-1}_\phi \left( \pi_{27}(\lie_X \phi) \right)\;.
    \end{align}
\end{subequations}
\end{lemma}
\subsubsection*{Nearly parallel $G_2$-identities and Dirac operator}
Nearly parallel $G_2$-structures enjoy a similar Hodge decomposition as Theorem \ref{hodgedecom6} for nearly Kähler structures. First, we list some useful identities for the exterior differential in terms of the irreducible representations.
\begin{proposition}\label{xfd7}
Let $f\in C^\infty$, $X\in \Omega^1$, $\beta_0\in \Omega^2_{14}$ and $\gamma_0\in \Omega^3_{27}$. We have 
\begin{multicols}{2}
    \begin{enumerate}[label=(\roman*)]
    \item $dX=\frac{1}{3}{\curl(X)} \lrcorner \phi + \pi_{14}(dX)$,
    \item $\curl(\curl(X))= d^*dX +4\curl(X)$,
    \item $\pi_7(d\beta_0) = \frac{1}{4} d^*\beta_0 $,
    \item $\pi_7(d^*\gamma_0)=\frac{4}{3}\pi_7(d\gamma_0)$ and 
    \end{enumerate}
\end{multicols}
\begin{enumerate}[label= (\roman*)] \setcounter{enumi}{4}
    \item $d*(X\wedge \phi)= \frac{4}{7}(d^*X)\psi +(\frac{1}{2}\curl(X)+X)\wedge \phi +2*i_\phi( \lie_X g)$.       
\end{enumerate}
\end{proposition}
\begin{proof}
The identities follow from differentiating the identities in \cite[Lemma 4.37]{SW10}, Lemma \ref{hitchintrick} and some algebraic yoga. A proof using coordinates of these identities can be found in \cite{SS20}. We only provide a coordinate-free proof of $(iv)$ using
Lemma \ref{hitchintrick}. For $X\in \Omega^1$, we have
\begin{align*}
4\int_M\langle X, \pi_7(d\gamma_0) \rangle\vol &= \int_M\langle X \wedge \phi, d\gamma_0 \rangle \vol= -\int_M d\gamma_0 \wedge (X \lrcorner \psi)= - \int_M \gamma_0 \wedge \lie_X \psi  \\
&=- \int_M \gamma_0 \wedge \J^{-1}\left(\lie_X \phi \right)
= \int_M \langle \gamma_0, d(X \lrcorner\phi) \rangle + \langle \gamma_0, X \lrcorner d\phi \rangle \vol \\
&=\int_M \langle d^*\gamma_0 , *(X \wedge \psi) \rangle \vol = 3\int_M\langle X, \pi_7(d^*\gamma_0)\rangle \vol \;,
\end{align*}
where $\J^{-1}$ just acted as $(-1)$ since $\gamma_0\in\Omega^3_{27}$. Since $X$ was arbitrary, the claim now follows.
\end{proof}
The identities $(iii)$ and $(iv)$ yield the following useful observation. 
\begin{corollary} \label{coclosed-type}
    Let $\beta_0 \in \Omega^2_{14}$ coclosed. Then $d\beta_0 \in \Omega^3_{27}$. Similarly, for $\gamma_0 \in \Omega^3_{27}$ coclosed,  $d\gamma_0 \in \Omega^4_{27}$.
\end{corollary}
\begin{proof}
 It only remains to verify that $\pi_1(d\beta_0)=0$ and $\pi_1(d\gamma_0)=0$. Indeed, by the Leibniz rule and Lemma \ref{formdecom7}, we have
 $$ 7\pi_1(d\beta_0) = d\beta_0 \wedge \psi = d(\beta_0 \wedge \psi) =0 ~~~~~~~~~~~~~~~~~~~~ 7 \pi_1(d\gamma_0)  =  d\gamma_0 \wedge \phi = d(\gamma_0 \wedge \phi) + \gamma_0 \wedge \psi =0 \;. ~~ \qedhere$$
\end{proof}
Let us now investigate the adapted Hodge decomposition. We do this by studying a suitably twisted Dirac operator, following the same strategy as Foscolo in \cite{Fos17}. In \cite{SS20}, Dwivedi and Singhal also use twisted Dirac operators to obtain a Hodge-like decomposition. The twisted Dirac here is different, and we obtain a different Hodge decomposition, which is more suitable for our purposes. 

Recall the choice of a $G_2$-structure is equivalent to the choice of a spin structure, together with the choice of a unit spinor. By the work of Bär\cite{Bar93}, the nearly parallel $G_2$ condition can be rephrased as the unit spinor $\Phi$ satisfying the real Killing spinor condition:
\begin{equation} \label{kspin7}
    \nabla_X \Phi = \frac{1}{2} X \cdot \Phi\;,
\end{equation}
where $\cdot$ denotes Clifford multiplication and $\nabla$ is the connection induced by the Levi-Civita connection on the spinor bundle. 

In terms of $G_2$-representations, we identify the real spinor bundle $\slashed{S}$ with $ \Lambda^0 \oplus \Lambda ^1$, where the isomorphism is given by 
$(f, X) \mapsto f\Phi + X \cdot \Phi$.
The Dirac operator $\slashed{D}$ under this isomorphism becomes
\begin{subequations}
\begin{align}
    \slashed{D}(f\Phi) &= -\frac{7}{2}f\Phi +\nabla f \cdot \Phi\;,\\
    \slashed{D}(X \cdot\Phi) &= \sum_{i=1}^7 e_i \nabla_{e_i}X \cdot \Phi - X \cdot \Phi - X \cdot \slashed{D}\Phi = dX \cdot \Phi + (d^*X) \Phi + \frac{5}{2}X \cdot \Phi \nonumber \\ &= d^*X\Phi + \left(\curl(X) +\frac{5}{2} X\right)\cdot \Phi  \;,   
\end{align}
\end{subequations}\;
where we used the identities from Proposition \ref{xfd7} and the fact that, the Clifford multiplication of  2-form $\beta=Y \lrcorner \phi + \beta_0$ is given by $\beta \cdot \Phi =  3 Y \cdot \Phi$ (cf. \cite[Sect. 4.2]{Kar08}).

Now, consider the operator
\begin{align*}
    \check{D}: \Omega^3_1 \oplus \Omega^3_7 & \rightarrow \Omega^4_1 \oplus \Omega^4_7 \\
    \gamma = ( f\phi, 2X \lrcorner \phi) & \mapsto \big( \pi_{1} (d \gamma) ,  \pi_{7} (d \gamma) \big) \;.
\end{align*}
Using the identities in Proposition \ref{xfd7}, we identify $\check{D}$ with the operator $D: \Omega^0 \oplus \Omega^1 \rightarrow \Omega^0 \oplus \Omega^1$ 
$$ D (f, X) =\big( \frac{8}{7}d^*X + 4f, df + \curl(X) +2X \big)\;.$$
First, notice that $D$ is an elliptic self-adjoint operator since $D$ and $\slashed{D}$ coincide up to rescaling and a self-adjoint term of order zero. We compute its kernel.
\begin{proposition}\label{kerDirac7}
Let $(M^7,\phi)$ be a complete nearly parallel $G_2$–manifold that is not isometric to the round 7–sphere. Then $\ker(D)=\{ X\in \Omega^1(M) \st \lie_X \psi=0 \}$. 
\end{proposition}
\begin{proof}
Let $(f, X) \in \ker(D)$. Then 
$$ d^*X = -\frac{7}{2}f ~~~~~~~~~~~~~~~~~~~~~~~~ df= -\curl(X)-2X\label{l72}\;. $$
Acting by $d^*$ on the right equation and combing with the left one,  we arrive at
$$ \Delta f = - 2 d^*X = 7 f\;.$$
By Obata's theorem, $f=0$ under the assumption that $(M^7, \phi)$ is not isometric to $ (S^7, g_{round})$. It remains to show that a vector satisfying $\curl(X)=-2X$ must preserve the $G_2$-structure. By $(ii)$ in Proposition \ref{xfd7}, we have 
$$\Delta X = d^*dX = \curl(\curl(X))- 4 \curl (X) =  12 X \;.$$
The Bochner-Weitzenböck identity on 1-forms implies $X$ is a Killing field, since $\operatorname{Ric}_g=6g$. Thus, by $(v)$ in Proposition \ref{xfd7}
$$\lie_X \psi =  d( X\lrcorner \psi) = -\frac{4}{7} (d^*X)\psi - \Big(\frac{1}{2}\curl(X) +X \Big) - 2 * i_\phi( \lie_X \g) = 0 \;. ~~ \qedhere $$
\end{proof}
\begin{remark}
For the round 7–sphere, the kernel of $D$ consists of elements of the form $(f, X - \nabla f)$, where $X$ satisfies $\lie_X \phi = 0$ and $f$ satisfies $\Delta f= 7f$.
\end{remark}
\begin{remark}
Contrary to the case of nearly Kähler structure, there might exist Killing fields that do not preserve the $G_2$-structure, for instance, when the associated metric cone has holonomy strictly contained in $\Spin(7)$.
\end{remark}
\begin{theorem}
    Let $(M^7, \psi)$ be a nearly parallel $G_2$-manifold that is not isometric to the round $7$-sphere, and denote by $\mathrm{K}$ the set of vector fields of $(M^7,\psi)$ that satisfy $\lie_X \psi=0$. Then the following holds.
    \begin{enumerate}[label=(\roman*)]
        \label{hodgedecom7}
        \item $\Omega^4 = \{X \wedge \varphi\st X \in \mathrm{K}\} \oplus d\Omega^3_{1\oplus 7} \oplus \Omega^4_{27}$. More concretely, for every $\chi\in \Omega^4$, there exists unique $X\in \mathrm{K}$, $Y\in \mathrm{K}^\perp$, $f\in \Omega^0$ and $\chi_0\in \Omega^4_{27}$ such that
        $$\chi = (X \wedge \varphi) + d( f \varphi + *(Y\wedge \varphi) + \chi _0\;, $$
        where $\mathrm{K}^\perp$ is the $L^2$-complement to $\mathrm{K}$.
        \item There is an $L^2$-orthogonal decomposition $\Omega^4_{exact} = d\Omega^3_{1\oplus 7} \oplus \Omega^4_{27, exact}$.
    \end{enumerate}
\end{theorem}
\begin{proof}
Statement (i) follows from the identification of $\check{D}$ with $\slashed{D}$ up to 0th order terms and Proposition \ref{kerDirac7}. Now, (ii) follows from (i). Notice that, for $X \in \mathrm{K}$, we have $d^*(X\wedge \phi) =- * \lie_X \psi =0 $, so $\{X \wedge \varphi\st X \in \mathrm{K}\}$ is $L^2$-orthogonal to exact forms and pointwise to $\Omega^4_{27}$. Orthogonality follows from Corollary \ref{coclosed-type}, in that if $\chi_0$ is closed, then $d^*\chi_0\in \Omega^3_{27}$.
\end{proof}
\subsection{The nearly parallel $G_2$ Hitchin functional}
For the remainder of this section, we assume that $M^7$ is a closed spinnable manifold, so it admits a $G_2$-structure. In $7$-dimensions, we have a non-degenerate quadratic form on $\Omega^4_{exact}$, defined as follows:
\begin{align*}
    Q: \Omega^4_{exact} & \rightarrow \R \\
        [d\gamma] & \mapsto  \int_M d\gamma \wedge \gamma  \;,
\end{align*}
induced by the isomorphism $(\Omega^4_{exact})^*\cong \Omega^4_{exact}$. 

As before, there is an open subset $\mathcal{V}= \Omega^4_{+}\cap \Omega^4_{exact}$ consisting of stable and exact 4-forms. Given a stable 4-form $\psi$ and a fixed orientation, we will consider the associated volume form $\vol_\psi = \frac{4}{7} \psi \wedge \widehat{\psi}$ and denote its Hitchin dual $\phi = \widehat{\psi}$. Comparing with the identity $\phi \wedge \psi= 7 \vol_g$, we get $\vol_g = \frac{1}{4}\vol_\psi$. In \cite{Hitchin01}, Hitchin introduced the functional 
\begin{align}\label{hitfun7}
    \mathcal{P}: \mathcal{V} & \rightarrow \R \nonumber \\
        \psi & \mapsto  \int_M \vol_\psi- 2 Q(\psi)\;,
\end{align}
and showed that its critical points correspond to nearly parallel $G_2$-structures. Indeed, we have
\begin{proposition}\label{EL7}
The Euler--Lagrange equation of $\mathcal{P}$ is $d \phi - 4 \psi =0 \;.$ In particular, critical points are nearly parallel $G_2$-structures. The gradient of $\mathcal{P}$ induced by $Q$ is given by $\partial_t \psi = d\phi -4\psi$. 
\end{proposition}
\begin{proof}
Let $\delta\psi= \chi= d\eta \in \Omega^4_{exact}$. Then $\delta \vol_\psi = \chi \wedge \widehat{\psi} = \chi \wedge \phi$ and so
$$\delta \mathcal{L}= \int_M \chi \wedge \phi -4 \int_M \eta \wedge \psi= \int_M \eta \wedge (d \phi -4\psi)\;. \qedhere$$
\end{proof}
We again have a nice geometric interpretation of the gradient flow in terms of the induced metric. 
\begin{proposition}\label{G2coneflow}
 Fix an orientation on $M$ and let $\psi(t)$, $t\in(a,b)$, be a family of stable exact $4$-forms and $g(t)$ the associated metric. Then the induced metric $\overline{g}=dr^2+r^2g\left((\log(r)\right)$ on $(e^a,e^b) \times M$ has holonomy contained in $\Spin(7)$ if and only if $\psi(t)$ satisfies the gradient flow equation of $\mathcal{P}$.
\end{proposition}
\begin{proof}
The condition of $\operatorname{Hol}(g_\phi)\subseteq \Spin(7)$ is equivalent to the 4-form $\Phi=dr\wedge r^3 \phi + r^4\psi$ being closed (and coclosed since it is self-dual). Thus, we get
\begin{equation*}
    0=d\Phi= - dr \wedge r^3 d_{M} \phi + 4 r^3 dr \wedge \psi + r^4 dr \wedge \frac{\partial \psi}{\partial r} ~~~~~~~ \implies  ~~~~~~~ r \frac{\partial \psi}{\partial r}= d_M \phi - 4 \psi\;,
\end{equation*}
where $d_{M}$ is just the restriction of the exterior differential $d$ along $\Lambda ^* T^*{M}$ and we used that $d_{M} \psi=0$. This is precisely the gradient flow equations under the change of variables $r=e^t$. The converse follows. 
\end{proof}
By replacing our Lagrange multiplier from $2$ to $2\lambda$ and considering the limit as $\lambda \rightarrow 0$ of the induced conformal metric $e^{2\lambda t}(dt^2 + g_\Sigma(t))$, we recover the result of Hitchin for $\Spin(7)$ metrics.

Similarly, we compute the second variation of $\mathcal{P}$. 
\begin{proposition}
Let $\chi_1, \chi_2 \in \Omega^4_{exact}$ and $\eta_i$ such that $d\eta_i=\chi_i$. The second variation of $\mathcal{L}$ with respect to $\chi_1, \chi_2$ is 
$$\delta^2 \mathcal{P}= \int_M (d\mathcal{J}\chi_2 - 4\chi_2)\wedge \eta_1\;.$$
In particular, the Hessian of $\mathcal{P}$ with respect to the indefinite metric induced by $Q$ is given by
$$H^\mathcal{P}(\chi)=  d \mathcal{J}\chi -4 \chi \;.$$
\end{proposition}
\begin{proof}
By Proposition $\ref{linear7}$, if $\delta\psi=\chi$, then $\delta \widehat{\psi}= \delta \phi =\mathcal{J}\chi$. Combining this with our formula for the first variation, and integrating by parts, the expression follows. 
\end{proof}
We want to study the spectral properties of $H^\mathcal{P}$. Since the functional $\mathcal{L}$ is invariant under $\operatorname{Diff}(M)$, it is convenient to work on a slice to the orbit of the diffeomorphism group. 

Let $(M^7, \psi)$ be a nearly parallel $G_2$-structure that is not isometric to the round $S^7$ and $\mathcal{O}$ be the orbit of $\operatorname{Diff}_0(M)$ in $ \Omega^4_{exact}$ going through $\psi$. The tangent space to this orbit is spanned by $\lie_X \psi$, for $X\in \mathrm{K}^\perp \subseteq \Omega^1$, the $L^2$-complement of vector fields infinitesimally preserving the $G_2$ metric. 

Using the Hodge decomposition of Theorem \ref{hodgedecom7}, we can parameterise $\chi \in \Omega^4_{exact}$ explicitly by 
$$\chi = \lie_X \psi + d(f\phi) +\chi_0\;;$$
with $f \in \Omega^0$, $X \in \mathrm{K}^\perp$, and $\chi_0 \in \Omega^4_{8,exact}$. In particular, it follows that taking $X=0$ defines a complement to the tangent space of the diffeomorphism action. As before, let  
$$\mathcal{W}=\big\{ d(f\phi) + \chi_0 \big\} \subseteq \Omega^4_{exact}\;,$$
for $f\in \Omega^0$ and $\chi_0\in \Omega^4_{27, exact}$. Taking the appropriate Hölder norm completions, we get
\begin{proposition}[\cite{Nor08} Theorem 3.1.4 and 3.1.7] \label{G2gaugeslice}
There exists a slice to the diffeomorphism group action in $\Omega^3_{exact} \times \Omega^4_{exact}$, whose tangent space is given by $\mathcal{W}$.
\end{proposition}
Going back to the study the spectral properties of $\mathcal{H}^\mathcal{P}$, we have
\begin{proposition}\label{hess7-1}
Assume $(M^7, \psi)$ is not isometric to the round $7$-sphere. Under the Hodge decomposition, the eigenvalue problem for the Hessian is equivalent to
\begin{subequations}\label{hessdecom7}
\begin{align}
    3f&= (\mu+4)f\;, \\
    \mu X - df &= 0\;, \\
    d*\chi_0 &= -(\mu+4)\chi_0\;.
\end{align}
\end{subequations} 
\end{proposition}
\begin{proof}
As above, let 
\begin{align*}
    \chi &=\lie_X \psi + d(f\phi) +\chi_0 
\end{align*}
with $f \in \Omega^0$, $X \in \mathrm{K}^\perp$ and $\chi_0 \in \Omega^4_{27,exact}$, in virtue of Theorem \ref{hodgedecom7}. By the definition of $\mathcal{J}$ and Lemma \ref{hitchintrick}, we get
$$\mathcal{J}\chi = \lie_X \phi+ *(df \wedge \psi) +3f\phi - *\chi_0 \;.$$
Now, since the $G_2$-structure is nearly parallel, we get 
\begin{align*}
    d\mathcal{J}\chi &= 4\lie_X \psi -  d(df \lrcorner \psi) +3d(f\phi) -d(*\chi_0) = 4\lie_{X-\frac{1}{4}df} \psi +3d(f\phi) -d(*\chi_0)  \;.
\end{align*}
Now, substituting this in $\mathcal{H}^\mathcal{P}$, and since the Hodge decomposition is orthogonal, we get the required system of equations \eqref{hessdecom7}.
\end{proof}
The case $\mu=0$ corresponds to the nullity of $\mathcal{H}^\mathcal{P}$, i.e. infinitesimal deformations of the nearly parallel $G_2$-structure. As expected, we recover the result of \cite{AS12} on infinitesimal deformations of nearly parallel $G_2$-structures (cf. \cite{NS21}). Notice that our functional approach does not detect the infinitesimal deformations arising from Killing fields that do not preserve the $G_2$-structure. That is, those arising from symmetries of the Sasaki-Einstein or $3$-Sasaki structures (cf. \cite{SS20}).
\subsection{The new $G_2$ Hitchin functional} \label{sectionnewG2}
We want to construct an analogue of the closed Hitchin functional. However, in this case, we cannot exploit any Hamiltonian-like behaviour as in the nearly Kähler case. Instead, we make a proposal imitating Proposition \ref{compareNK_EH}.

Recall that $\mathcal{V}$ is the space of stable exact 4-forms in $M^7$. Given a fixed orientation on $M$, the 4-form defines a $G_2$-structure on $M$, with torsion $d\phi = \tau_0 \psi + * \tau_3$. We define the Hitchin functional 
\begin{align*}
    \mathcal{T}:  \mathcal{V}\rightarrow & \R \\
     \psi  \mapsto  & \int_M \frac{7 \tau_0^2- 5}{4} \vol_\psi \;.
\end{align*}
\begin{proposition}\label{compareG2_EH}
    Let $M^7$ be a spinnable manifold with a fixed orientation, and $G:\mathcal{V}\rightarrow \operatorname{Met}(M)$, mapping the $G_2$-structure to its underlying metric. Consider $\widehat{\mathcal{S}}= G^*(\mathcal{S})$ the pullback of the Einstein--Hilbert action. The Hitchin functional $\mathcal{T}$ satisfies $\mathcal{T} \geq \widehat{\mathcal{S}}$, with equality if and only if the $G_2$-structure is a constant multiple of a nearly parallel $G_2$ metric. 
\end{proposition}
\begin{proof}
    By Lemma \ref{scalar_curv_G2}, we have 
    $$\widehat{\mathcal{S}}=  \frac{1}{6} \int_M \left(42\tau_0^2 - \frac{1}{2}\abs{\widehat{\tau_2}}^2\right)- 5\vol_g = \int_M 7\tau_0^2 - 5 - \frac{1}{12}\abs{\widehat{\tau_2}}^2 \vol_g\;.$$
    Using the relation $7\vol_g = \phi \wedge \psi = \frac{7}{4}\vol_\psi$, the claim follows.
\end{proof}
Let us study the variations of $\mathcal{T}$. We have
\begin{proposition}
    The Euler--Lagrange equation of $\mathcal{T}$ is
\begin{equation}
    \tau_0\mathcal{J}d\phi+ \frac{1}{2}\mathcal{J}(d\tau_0\wedge \phi) -\frac{7\tau_0^2+5}{4}\phi =0
\end{equation}
    In particular, the critical points of $\mathcal{T}$ are nearly parallel $G_2$-structures, up to orientation. The gradient flow with respect to the quadratic form $Q$ is 
    \begin{equation} \label{npg2flow}
        \partial_t \psi = d\left[\tau_0\mathcal{J}d\phi + \frac{1}{2}\mathcal{J}(d\tau_0\wedge \phi) - \frac{7\tau_0^2+5}{4} \phi \right] \;.
    \end{equation}
\end{proposition}
First, we need the following technical result
\begin{lemma}\label{variation_tau0}
    The variation of $\tau_0$ along $\delta \psi = \chi$ is 
     \begin{equation} \label{variationtau0} \delta \tau_0 \vol_\psi  = \frac{1}{7}\left[d( \mathcal{J}\chi \wedge \phi) + 2 d\phi \wedge \mathcal{J}\chi\right] - \tau_0 \phi \wedge \chi \;. 
     \end{equation}
\end{lemma}
\begin{proof} Let $\delta \psi =  \chi$. Then $\delta \varphi = \mathcal{J}\chi$ by Proposition \ref{linear7}. Let us compute the variation of $\tau_0$. By definition, we have $d\phi \wedge \phi = 4\tau_0 \psi \wedge \phi = 7 \tau_0 \vol_\psi $. Taking the variation of this identity, we get
    $$
   7 \delta \tau_0 \vol_\psi + 7 \tau_0 \phi \wedge \chi  = d \mathcal{J}\chi \wedge \phi + d\phi \wedge \mathcal{J}\chi\;.
    $$
By the Leibniz rule, the claim follows.
\end{proof}
\begin{proof}[Proof of Proposition]
Using the result above, we have 
    \begin{align}
        \delta \mathcal{T} &= \frac{1}{4}\int_M 14 \tau_0 \delta \tau_0 \vol_\psi + \left(7\tau_0^2- 5\right) \phi \wedge \chi  \nonumber \\
        &= \frac{1}{4}\int_M 4 \tau_0   d\phi \wedge \mathcal{J}\chi -2 d\tau_0 \wedge \mathcal{J}\chi \wedge \phi - \left(7\tau_0^2+ 5 \right) \phi \wedge \chi \nonumber \\
        &=\int_M \chi \wedge \left(\tau_0\mathcal{J}d\phi + \frac{1}{2}\mathcal{J}(d\tau_0\wedge \phi)- \frac{7\tau_0^2+5}{4} \phi \right)\;, \label{1stvariationnewG2}
    \end{align}
and the Euler--Lagrange equation follows. Let's study critical points. Using $d\phi=4 \tau_0 \psi +*\tau_3$, it is clear that $\tau_3=0$ and $\tau_0=C \in \R$. We get an equation for $\tau_0$:
 $$ 12 \tau_0^2 - \left(7\tau^2_0 +5\right) =0\;,$$
 with solutions $\tau_0=\pm 1$. The case $\tau_0=1$ is the nearly parallel $G_2$ condition. For $\tau_0= -1$, we obtain a nearly parallel $G_2$-structure for the reversed orientation. The formula for the gradient flow follows from taking $\chi= d\eta$ and integrating by parts.
\end{proof}
Notice that, unlike the case of nearly Kähler structures, the flow depends explicitly on the torsion $\tau_0$ and its derivatives. In particular, the flow is third order in $\psi$ and thus non-parabolic.

Before studying the second variation, we have the following technical computation
\begin{lemma}
Let $(M, \psi)$ be a nearly parallel $G_2$-structure, and consider a variation $\delta \psi =  \chi =  f\psi +X \wedge \phi +  \chi_0$. We have
\begin{align*}
    \delta \tau_0 & = \frac{1}{7}d^*X - \frac{1}{4} f \;.
\end{align*}
\end{lemma}
\begin{proof}
    From Equation \eqref{variationtau0}, we get
    \begin{align*}
        \delta \tau_0 \vol_\psi& = \frac{1}{7}\big[d(\mathcal{J}\chi \wedge \phi) + 2\mathcal{J}(4\psi)\wedge \chi\big] - \phi\wedge \chi = \frac{1}{7}\left(d(\mathcal{J}\chi \wedge \phi) + 6\phi \wedge \chi\right)  - \phi\wedge \chi\\
        &=\frac{1}{7}\left[d(*(X \wedge \phi) \wedge\phi) -  f \psi \wedge \phi \right] = -\frac{1}{7} \left(4d*X + f \psi \wedge \phi \right) = \left(\frac{1}{7}d^*X - \frac{1}{4} f \right) \vol_\psi\;,
    \end{align*}
    where we used the relation $7 \vol_g = \phi \wedge \psi = \frac{7}{4} \vol_\psi$.
\end{proof}
\begin{proposition}
    The second variation of $\mathcal{T}$ along $\delta\psi _i = \chi_i = f_i\psi + X_i \wedge \phi + (\chi_0)_i$ is given by
    \begin{equation}\label{2ndvariationNK2_eq}
        \delta^2 \mathcal{T}= \int_M  \chi_1 \wedge \left[ \mathcal{J}d \mathcal{J}\chi_2 - 4 \mathcal{J}\chi_2 + \frac{1}{14}* \left[d\left(d^*X_2 - \frac{7}{4} f_2\right)\wedge \phi \right] -\frac{1}{14} \left(d^*X_2 - \frac{7}{4} f_2\right) \phi \right] \;.
    \end{equation}
    In particular, the Hessian with respect to $Q$ is given by
    $$\mathcal{H}^\mathcal{T}(\chi)=  d\left[ \mathcal{J}d \mathcal{J}\chi - 4 \mathcal{J}\chi+\frac{1}{14}* \left[d\left(d^*X - \frac{7}{4} f\right)\wedge \phi \right] -\frac{1}{14} \left(d^*X - \frac{7}{4} f\right) \phi \right]\;.$$
\end{proposition}
\begin{proof}
Notice that directly taking the variation of \eqref{1stvariationnewG2} would require us to understand $\delta \mathcal{J}$. We avoid this by noticing that we can rewrite $\delta \mathcal{T}$ as
$$\delta \mathcal{T}= \int_M \mathcal{J}\chi \wedge \left(\tau_0 d\phi + \frac{1}{2}(d\tau_0\wedge \phi)- \frac{7\tau_0^2+5}{4} \mathcal{J}^{-1}\phi \right) = \int_M \mathcal{J}\chi \wedge \left(\tau_0 d\phi + \frac{1}{2}(d\tau_0\wedge \phi)- \frac{7\tau_0^2+5}{3} \psi \right).$$
The right hand side can thus viewed as the variation of $\mathcal{T}$ for $\delta \phi = \delta \widehat{\psi} =\mathcal{J}\chi \in \Omega^3$. Thus, 
\begin{align*}
\delta^2 \mathcal{T}& = \int_M \mathcal{J}\chi_1 \wedge \left[\tau_0 d \mathcal{J}\chi_2 + \delta \tau_0 d\phi + \frac{1}{2}(d \delta_0 \tau_0\wedge \phi)-  \frac{14}{3} \tau_0 \delta \tau_0 \psi- \frac{7\tau_0^2+5}{3} \chi_2 \right] \\ 
&=\int_M \mathcal{J}\chi_1 \wedge \left[d \mathcal{J}\chi_2 - \frac{2}{3} \delta \tau_0 \psi + \frac{1}{2}(d \delta \tau_0 \wedge \phi) - \frac{7\tau_0^2+5}{3} \chi_2 \right]  \;.
\end{align*}
Using the lemma, we can rewrite this as \eqref{2ndvariationNK2_eq}. From the definition of $Q$, the expression of the Hessian is straightforward.
\end{proof}
Let us study the Hessian's spectrum. By Proposition \ref{G2gaugeslice}, we can restrict ourselves to the tangent of a slice to the diffeomorphism orbit $\mathcal{W}= \{ d(f\phi) + \chi_0\} \subseteq \Omega^4_{exact}$. We have
\begin{proposition}\label{hessdecomposition_G2}
Let $(M,\psi)$ be a nearly parallel $G_2$-manifold that is not isometric to the round $S^7$. The eigenvalue problem $(\mathcal{H}^\mathcal{T} - \mu):\mathcal{W}\rightarrow \mathcal{W}$ is equivalent to the PDE
\begin{align}
    \Delta \chi_0 +4 d* \chi_0  =\mu \chi_0 \;
\end{align}
when $\mu \neq -3$. For $\mu=-3$, the eigenforms are additionally given by multiples of $\psi$.
\end{proposition}
\begin{proof}
As in the proof of Proposition \ref{hess7-1}, we can take $\chi\in \mathcal{W}\subseteq \Omega^4_{exact}$ as $$\chi= d(f\phi) +\chi_0 = 4f\psi + df\wedge \phi + \chi_0$$ with $f \in \Omega^0$ and $\chi_0 \in \Omega^4_{27,exact}$, in virtue of Theorem \ref{hodgedecom7}, and Proposition \ref{G2gaugeslice}. We compute the four terms of the second variation separately. First,
\begin{align*}
    d\mathcal{J}\chi &= d \left( *(df \wedge \phi) +3f\phi - *\chi_0 \right)=3d(f\phi)-  d(df \lrcorner \psi) -d(*\chi_0)  = 3d(f\phi) - \lie_{df} \psi -d(*\chi_0)\;.
\end{align*}
Thus, we have 
\begin{align*}
 d \mathcal{J}d\mathcal{J}\chi &= d\mathcal{J}\left[3d(f\phi) - \lie_{df} \psi -d(*\chi_0)\right]=d \left[3*(df\wedge \phi)+9 f\phi  - \lie_{df} \phi+*d*\chi_0 \right]\\ &= 9 d(f\phi) -7\lie_{df}\phi + \Delta \chi_0\;.
\end{align*}
Similarly, using the identity $d* \left(X \wedge \phi\right) = - \lie_X \psi$ once more, the third term in \eqref{2ndvariationNK2_eq} becomes
$$ \frac{1}{14} d*\left[d \left(\Delta f - 7f\right)\wedge
        \phi \right] =  - \frac{1}{14} \lie_{d(\Delta f - 7f)}\psi \;.$$
The fourth term is simply $-\frac{1}{14}d\left[(\Delta f - 7 f)\phi \right]$. Putting all of these together and using that the Hodge decomposition of \ref{hodgedecom7} is orthogonal, we have 
\begin{align*}
     - \frac{1}{14}\left(\Delta f - 7f\right) -3f &= \mu f  \\
     - \frac{1}{14}d\left(\Delta f - 7f\right) - 3df& = 0\\
     \Delta \chi_0 +4 d* \chi_0  & =\mu \chi_0 \;,
\end{align*}
Now, if $df\neq 0$, the first two equations combine to yield $\mu=0$, which implies $f=0$ since $(\Delta -7)$ is strictly positive, by Obata's theorem \cite{Ob62}.  If $f=C \in \R$, it follows that $\mu=-3$.
\end{proof}
\begin{definition}\label{G2index}
    Let $(M^7, \phi)$ a nearly parallel $G_2$-manifold. We define the index of the nearly parallel $G_2$-structure $\operatorname{Ind}_{\phi}$ as the number of negative eigenvalues of the Hessian endomorphism $\mathcal{H}^\mathcal{T}: \mathcal{W}\rightarrow \mathcal{W}$ at $\psi$ minus one. 
\end{definition}
First, the following lemma shows that the index is well-defined.
\begin{lemma}\label{lemmaindex}
The spectrum of $\mathcal{H}^\mathcal{T}$ is bounded below by $-4$. In particular, the index is well-defined.
\end{lemma}
\begin{proof}
Let $\chi\in \Omega^4_{exact}$. If $\mu \neq -3$, we know that $\chi \in \Omega^4_{27}$, by Proposition \ref{hessdecomposition_G2}. Taking the $L^2$-norm of $d*\chi+ 2\chi$, we get 
$$0 \leq \langle d*\chi+ 2\chi, d*\chi+ 2 \chi\rangle = \langle d*d*\chi+ 4 d*\chi, \chi\rangle + 4 \norm{\chi}^2=  \langle \mathcal{H}^\mathcal{T}(\chi), \chi \rangle + 4 \norm{\chi}^2\;. \qedhere$$
\end{proof}
Moreover, we have a relation between the spectrum of the Hessians $\mathcal{H}^\mathcal{P}$ and $\mathcal{H}^\mathcal{T}$:
\begin{proposition}\label{2nd21st}
Solutions to $(\mathcal{H}^\mathcal{P}-\lambda)\chi=0$  with $\lambda \in \R$ for $\chi \in \Omega^4_{exact}$ are in correspondence with solutions to $\mathcal{H}^\mathcal{T}(\chi)=\mu \chi$ with $\mu \geq -4$. Moreover, the range $\lambda \in (-4,0)$ is in two-to-one correspondence with the range $\mu\in (-4,0)$, excluding multiples of the $4$-form $\psi$. 
\end{proposition}
\begin{proof}
First, $\chi=C\psi$ for $C\in \R$ are solutions to $(H^\mathcal{P}-\lambda)\chi=0$ for $\lambda=-1$ and to $(H^\mathcal{T}-\mu)\chi=0$ for $\mu=-3$. Thus, we can now assume that $\chi \in \Omega^4_{27,exact}$. 

First, let $\chi$  be a solution to $d*\chi = -(\lambda+4) \chi$. Then 
$$\mathcal{H}^\mathcal{T}(\chi)= - (\lambda+4) d*\chi - 4(\lambda+4) \chi = \big[ (\lambda+4)^2 - 4(\lambda+4)\big] \chi\;,$$ which is negative in the interval $\lambda \in (-4,0)$. Conversely, assume $\chi$ satisfies $\mathcal{H}^\mathcal{T}(\chi)=\mu \chi$ with $\mu\geq-4$. Let $\gamma_\pm=d*\chi - \lambda \chi$, for $\lambda_\pm= -2 \pm \sqrt{\mu+4}$. Clearly, $\gamma \in \Omega^4_{27,exact}$. If $\gamma_\pm=0$, we are done. Otherwise, $\gamma_\pm$ is a non-trivial solution to $\mathcal{H}^\mathcal{P}-\lambda$. By substituting $\gamma$ in $\mathcal{H}^\mathcal{T}(\chi)-\mu\chi$, we have
$$0=\Delta \chi +4d* \chi- \mu\chi = d*(\gamma + \lambda \chi) +4(\gamma + \lambda \chi)- \mu\chi= d*\gamma + (\lambda+4)\gamma + (\lambda^2 + 4\lambda - \mu)\chi \;. $$
For our chosen values of $\lambda$, the rightmost term cancels and we have that $\gamma$ satisfies $\mathcal{H}^\mathcal{P}(\gamma)=\lambda \gamma$, as needed.
\end{proof}
In particular, we have
\begin{corollary}\label{Qindex_G2}
    Let $(M, \phi)$ be a nearly parallel $G_2$-manifold, and consider the spaces $\mathcal{E}(\lambda)=\left\{ \chi_0 \in \Omega^4_{27} \st d*\chi_0 = \lambda \chi_0 \right\}$. The Hitchin index of the nearly parallel $G_2$-structure is given by
    \begin{equation*}
        \operatorname{Ind}_\phi = \sum_{\lambda \in (-4,0)} \dim~ \mathcal{E}(\lambda) \;.
    \end{equation*}
\end{corollary}
One could try to relate this with the Morse co-index of $\mathcal{T}$, as we did for the Hitchin index of nearly Kähler structures in Proposition \ref{index_vs_coindex}. However, a moment of thought suffices to realise that both the index and the co-index of $\mathcal{T}$ are infinite. Indeed, taking $\mathcal{E}$ as above, we have
$$ \delta^2 \mathcal{T} \Bigr|_{\mathcal{E}(\lambda)} = \int_M 
\chi_1 \wedge \left[(\lambda +4)* \chi_2 \right] = (\lambda +4)  \langle \chi_1, \chi_2 \rangle \;.$$
As for nearly Kähler structures, one could investigate the index of most known examples since they all possess some symmetry that would allow us to reduce the PDE to a simpler problem. We do not work out any examples but provide an outline of how to compute or bound the Hitchin index. 
\begin{enumerate}[label=(\roman*)]
    \item \textbf{Homogeneous examples:} The Peter-Weyl formalism for reductive spaces described above carries over verbatim. The case of nearly parallel $G_2$-structure is slightly more challenging since the differential operator is not simply a Laplacian, thus computations become more tedious. 

    Some computations in this direction were carried out by Alexandrov and Semmelmann in \cite{AS12} and Lehmann \cite{Leh21}. 

    \item \textbf{Sasaki-Einstein examples:} Recall that the inclusion $\SU(4) \subset \Spin(7)$ implies that every Sasaki-Einstein manifold carries a natural nearly parallel $G_2$-structure.  Let us assume that the underlying Sasaki structure is quasi-regular, so the Reeb field integrates to an $S^1$ action. In this case, the PDE $$\Delta \chi + 4d^* \chi = \mu \chi $$
    is $S^1$-equivariant. Thus, one can try to use the Peter-Weyl (Fourier) formalism along the fibres to reduce this problem to a complex PDE on the leaf space and obtain a bound for the index in terms of Hodge numbers of the complex orbifold base, using the results of Nagy's PhD thesis \cite{Nagy01}.

    The added difficulty in this case is that, while the PDE above is $S^1$-invariant, the underlying $G_2$-structure is not (and thus neither is $\Omega^4_{27}$), so one would need to check that the forms constructed above had the correct type. 

    \item \textbf{Squashed examples:} The squashed nearly parallel $G_2$ metrics are constructed by rescaling the fibres of a 3-Sasaki manifold. In particular, the squashed metric has an isometric action by $\SU(2)$, with a 4-dimensional orbifold leaf space.
    
    Thus, one can follow the same strategy of reducing the PDE to the 4-orbifold by using the Peter-Weyl formalism along the fibres. Similar ideas appeared in a recent preprint of Nagy and Semmelmann \cite{NS23}.
\end{enumerate}
\section{Outlook}
 We outline the connection between the nearly the Hitchin functionals in the study of $G_2$-conifolds 
 following \cite{LK20}, focusing on the Hitchin index. We conjecture there is an analogue discussion for $\Spin(7)$-conifolds.

The expectation is that the Hitchin acts as the stability index for conically singular $G_2$ manifolds. That is, the index measures the codimension of the singularity in the moduli space of conically singular $G_2$ manifolds. A first indication of this is the dimension bound of the obstruction space for $G_2$-conifold deformation:
\begin{proposition}[{\cite[Prop. 6.11]{LK20}}]
    Let $(M, \phi)$ be a conically singular $G_2$-manifold with singularities $p_1, \dots,p_n$, modeled on the $\Sigma_1, \dots, \Sigma_n$. The dimension of the obstruction space to the deformation problem is bounded above by 
$$\dim (\mathcal{O}_{/\mathcal{W}}) \leq n-1 + \sum_{i=1}^n \left(\Ind^{\Sigma_i} \right);.$$ 
Moreover, if $\Ind^{\Sigma_i}=0$ for all $i$, the remaining obstruction space is ineffective.
\end{proposition}
To pursue this discussion, it would be useful to study manifolds with both asymptotically conical (AC) and conically singular (CS) ends. Although this case is not directly addressed by Karigiannis and Lotay, their methods should extend with minimal difficulty. In particular, for a manifold with only one conically singular point and an asymptotically conical end, we expect that the virtual dimension of the moduli space will be given by the difference of the Hitchin indices. 

This expectation can be motivated by treating $\mathcal{L}$ as an analogue of the Chern--Simons functional in instanton Floer theory. Consider a family of $\SU(3)$-structures $\left(\rho(t), \sigma(t)\right)$ on $\Sigma$, evolving with the gradient flow of $\mathcal{L}$ and connecting two of its critical points. Proposition \ref{SU3coneflow} implies there is an associated $G_2$ conifold with one CS and an AC end. Following the Chern--Simons analogy, the virtual dimension of the moduli space of such conifold should be given by the spectral flow of the family of $\SU(3)$-structures. In view of Proposition \ref{index2index} (cf. Prop. \ref{indexaslaplaceeigenvalueNK}), this corresponds to the index difference of the two nearly Kähler structures.  

\appendix

\section{Non parabolicity of the nearly Kähler Laplacian flow} \label{shortenu}
We show that the nearly Kähler gradient flow introduced in Section \ref{sectionclosedSU3} is not strictly parabolic, even after using a DeTurck-type trick. Thus, using standard techniques, one cannot guarantee the short-time existence and uniqueness of solutions to the flow. In particular, the symbol of the nearly Kähler Laplacian flow \eqref{lapflow} takes a similar shape to the $G_2$ Laplacian coflow, introduced by Karigiannis, McKay and Tsui in \cite{KMT12} (cf. \cite{Gri13}).

We begin by constructing suitable DeTurck vector fields, following the exposition of \cite{BX11}. We then compute the symbol of the flow, modified by a suitable DeTurck field. 

\subsection{The DeTurck vector fields}
We use the same recipe that DeTurck used for the Ricci flow, or Bryant and Xu \cite{BX11} for the $G_2$ Laplacian flow. Let $M$ be a manifold and $g$ a metric, $\nabla$ its Levi-Civita connection and $\nabla^0$ a fixed torsion-free connection (e.g. the Levi-Civita of a background metric). The difference 
$$T = \nabla- \nabla^0$$
is a well-defined section of $\operatorname{Sym}^2 TM^* \otimes TM$. Identifying $TM$ with $TM^*$ via the metric, and using the decomposition $\operatorname{Sym^2}TM = TM \oplus \operatorname{Sym}^2_0 TM$, we view $T$ as a section of $TM \oplus TM \otimes \Sym^2_0 TM$.

We obtain two vector fields from $T$, one from the first term of the decomposition, labelled $V_1$, and the other by contracting $TM$ with $\operatorname{Sym}^2_0 TM$, labelled $V_2$. Thus, whenever we have a $G$-structure on $TM$ with $G\subseteq \SO(n)$, we will have (at least) a two-dimensional family of vector fields associated with it.

Let us study the linearisation of these vector fields. Since $\nabla^0$ is fixed, the linearisation $T$ depends exclusively on the variation of the metric $g$. Let $h= \frac{d}{dt} g_t \Bigr|_{t=0}$. By the Koszul formula, one gets
$$g\left(T_*(h) \left(X,Y\right), Z \right) =  \frac{1}{2}\left( \nabla_X(Y, Z) + \nabla_Y(X,Z) - \nabla_Z (X,Y) \right) \;.$$
Using the trace and traceless decomposition for $h$, $h= fg + h_0$, we get (cf. \cite[Sect. 2.2]{BX11})  
$$ {V_1}_*(h) = \operatorname{grad}(f)  ~~~~~~~  {V_2}_*(h)= \operatorname{div}(h_0)\;.$$

In our case of interest, the $\SU(3)$-structure induces the further decomposition $\Sym^2_0\cong \Sym^2_+ \oplus \Sym^2_- $ into traceless $J$-invariant and $J$-anti-invariant symmetric maps. Thus, we obtain a 3-dimensional family of suitable DeTurck vector fields. We only consider the trace and the $J$-invariant vector fields for order reasons.

Fix $(\rho, \sigma)$ an $\SU(3)$-structure. Using the isomorphisms from Lemma \ref{decompositionSU(3)}, $\Lambda^4_1 \cong \R$ and $\operatorname{Sym}^2_+ \cong\Lambda^5_8$, it follows that there exists a universal constant $A$ such that a variation of $\sigma$, $\delta \sigma= f \sigma + X \wedge \hat{\rho} +\chi_0$, the induced variation of the metric is given by 
$$\delta g = \frac{1}{2}f g + A \iota^{-1}(\chi_0).$$
We need to compute the $\div\left(\iota^{-1}(\chi_0)\right)$. We have the following lemma.
\begin{lemma} Let $\chi \in \Lambda^4_8$. There is a universal constant $B$ for which 
    $$ d \chi = B *\div \left(\iota^{-1}(\chi)\right) +l.o.t\;,$$
    where $l.o.t$ is some 5-form depending smoothly on $\chi$ and the torsion of the $\SU(3)$-structure.
\end{lemma}
\begin{proof}
Assume $M$ carries a torsion-free $\SU(3)$-structure (i.e. Calabi-Yau). Consider the diagram
$$
\Lambda^1 \xleftarrow{~~c~~} \operatorname{Sym}_+^2 TM \otimes \Lambda^1 \xrightarrow{\mathbf{\iota} \otimes 1} \Lambda_{8}^4 \otimes \Lambda^1 \xrightarrow{\operatorname{Alt}} \Lambda^5  \;,$$
where $\operatorname{Alt}$ denotes skewsymmetrization, $c$ denotes contraction by the metric and $\iota$ is the map given in Lemma \ref{decompositionSU(3)}. By Schur’s lemma, there exists a universal constant $B$ such that
$$ * \operatorname{Alt}( \chi \otimes \alpha) =  B c \left( \iota^{-1}(\chi) \otimes \alpha\right) \;.$$
Now, we have $d(\chi) = \operatorname{Alt} \circ \nabla(\chi)$ and 
since $\nabla$ preserves the Calabi-Yau structure, it commutes with the map $\iota$, proving the statement for the torsion-free case. The torsion of the $\SU(3)$-structure will modify the identity involving only zeroth order terms.
\end{proof}
\begin{lemma}
Let $(\rho, \sigma)$ be an $\SU(3)$-structure. Then the DeTurck procedure outlined above allows us to construct two vector fields, $W_1(\rho,\sigma)$ and $W_2(\rho,\sigma)$, depending smoothly on the $\SU(3)$-structure, whose linearisation along a variation $\delta\sigma= \chi= f \sigma +  X \wedge \hat{\rho} + \chi_0$ is given by
$$ {W_1}_*(\chi) = df  ~~~~~~~~~~~~~~~~~~~~~~~  {W_2}_*(\chi) = *d\chi_0 +l.o.t\;.$$
\end{lemma}
We have rescaled our vector fields to eliminate all the constants and lighten notation. Since they are universal, there are no ambiguities in us doing so. If we restrict ourselves to $\SU(3)$-structures where the 4-form $\sigma=\frac{\omega^2}{2}$ is closed, we can further rewrite our DeTurck fields.
\begin{proposition}\label{dTTsu3}
    Let $(\rho, \sigma)$ be a $\SU(3)$-structure such that $d\sigma=0$. Then, the DeTurck fields can be chosen as 
    $$ {V_1}_*(\chi) = df  ~~~~~~~~~~~~~~~~~~~~~~~  {V_2}_*(\chi) = J\curl(X) +l.o.t\;.$$
\end{proposition}
\begin{proof}
We need to prove that a linear combination of ${W_1}_*$ and ${W_2}_*$ is equal to ${V_2}_*$, up to zeroth order terms. Linearising the condition $d\sigma=0$, we have
$$d\chi = df \wedge \sigma - J\curl(X) \wedge \sigma + d \chi_0 + l.o.t = 0\;.$$
Thus, we need to take $V_2= W_2 +W_1$.
\end{proof}
\subsection{The nearly Kähler Laplacian flow}
Let us study the parabolicity of the nearly Kähler Laplacian flow \eqref{SU3flow}, modified by the DeTurck term:
\begin{equation} \label{SU3flow}
    \begin{cases}
        \partial_t \sigma =  \Delta_\sigma\sigma -12 \sigma  + \lie_{V(\sigma)} \sigma \\
        d\sigma=0 \\
        \sigma(0)=\sigma_0\;,
    \end{cases}
\end{equation}
for $V(\sigma) = 3 V_1(\sigma) + 2 V_2(\sigma)$, with $V_i$ given in Proposition \ref{dTTsu3}. 

First, recall the following identity on nearly Kähler structures:
\begin{lemma}[{\cite[Lemma 3.7]{Fos17}}] \label{useful6}
Consider $(M, \omega, \rho)$ a nearly  structure and let $X\in \Omega^1$. Then, 
    $$d(X \lrcorner \rho) - d^*(X \wedge \rho)= (J\curl(X) + 2X) \wedge \omega - d^*X\rho - d^*(JX)\widehat{\rho} \;.$$
\end{lemma}
We compute the linearisation of $P=\Delta_\sigma \sigma - 12\sigma + \lie_{V(\sigma)}\sigma$ along $\chi= f\sigma + X \wedge \hat{\rho}+ \chi_0 \in \Omega^4_{closed}$. Since $\sigma$ is closed, $\Delta_\sigma \sigma = -d*d* \sigma$ and $\lie_V\sigma = d(V \lrcorner \sigma)$ and so 
\begin{align*}
    D_\sigma P(\chi) &= - d* d \mathcal{K} \chi + d( V_*(\chi) \lrcorner \sigma)  = -d *d(2f \omega + X \lrcorner \rho - *\chi_0) + d\left(( 3 {V_1}_* + 2{V_2}_*)\lrcorner \sigma\right) \;.
\end{align*}
Similarly, we can compute $\Delta_\sigma \chi = dd^* \chi = -d *d (f \omega + X \lrcorner \rho + *\chi_0)$.
And so, by Lemma \ref{useful6}, we have 
\begin{align*}
   D_\sigma P(\chi)+ \Delta \chi=& -d * d( 3f \omega + 2X \lrcorner \rho) + d\left(( 3 {V_1}_* + 2{V_2}_*)\lrcorner \sigma\right)   \\
   =& -d * \left[ 3 df \wedge \omega  + 2 \left( J\curl(X) \wedge \omega - d^*X \rho - d^*(JX) \hat{\rho} +d^*(JX \wedge \hat{\rho}\right) \right] \\
   &+ d\left( 3Jdf - 2\curl(X) \right) \wedge \omega  +l.o.t.  \\
   =& -d \left[  (3 Jdf -2\curl(X)) \wedge \omega  - 2 d^*X \hat{\rho} + 2d^*(JX) \rho\right] +  d\left( 3Jdf - 2\curl(X) \right) \wedge \omega+ l.o.t.  \\
   =& 2 \left[ dd^*X - Jdd^*(JX) \right] \wedge \hat{\rho} + l.o.t\;.
\end{align*} 
In particular, we have proved
\begin{proposition}\label{nonparabolicity}
The linearization of the operator $P=\Delta_\sigma \sigma +\lie_V(\sigma) \sigma$ evaluated at a closed 4-form $\chi= f\sigma + X \wedge \hat{\rho} + \chi_0$
is given by
$$D_\sigma P (\chi) = -\Delta_\sigma \chi +2(dd^*X - Jdd^*JX)  \wedge \hat{\rho}+ dF(\chi)$$
where $F(\chi)$ is a 3-form-valued algebraic function of $\chi$ that depends on the torsion of the $\SU(3)$-structure. In particular, its principal symbol in the direction $\xi$ satisfies 
$$ \langle \operatorname{S}_\xi (D_\sigma P)(\chi), \chi \rangle = - \abs{\xi}^2 \abs{\chi}^2 + 4 \left(\langle \xi, X \rangle^2 + \langle \xi, JX \rangle^2 \right) \;,$$
which is not coercive, so the flow is not parabolic.
\end{proposition}
\begin{proof}
Only the symbol computation remains. Since we know that $\operatorname{S}_\xi (d) = \xi \wedge$ and $\operatorname{S}_\xi (d^*) = \xi \lrcorner$, the computation follows from the identity $\langle X \wedge \hat{\rho} , Y \wedge \hat{\rho} \rangle = 2 \langle X , Y \rangle$.
\end{proof}
We notice a couple of remarks. First, the term $dd^*X - Jdd^*(JX)$ can not be reabsorbed by an additional term of the shape $\lie_{W(\sigma)} \sigma$ for a different choice of field $W(\sigma)$. Indeed, the linearized operator for $W(\sigma)\lrcorner \sigma$ along $\chi=f\sigma +X \wedge \hat{\rho}+ \chi_0$ will be a linear combination of $\curl(X)$, $\curl(JX)$, $df$ and $Jdf$, plus lower order terms. Second, one could attempt to modify the flow to make it elliptic, following the construction of Grigorian's modified $G_2$ Laplacian coflow in \cite{Gri13}. The idea is to construct second-order operators depending on $\sigma$, whose linearisation cancels out the terms $dd^*X$ and $Jdd^*(JX)$. In that direction, we have a first partial result.

Recall that $\tau_0(\sigma)= \frac{1}{3} *(d\omega \wedge \hat{\rho})$ is the 1-dimensional part of the torsion of $\sigma$, so it satisfies
 \begin{equation} \label{tau0_vol}
     \frac{1}{3}\sigma \wedge \omega = \frac{\tau_0^{-2}}{4} d\omega \wedge \widehat{d\omega}\;.
 \end{equation}
 \begin{lemma}
 The first order variation along $\chi=f\sigma+ X \wedge\hat{\rho}+\chi_0$ of $\tau_0$ is given by
$$\partial_{\chi} \tau_0 = -\frac{1}{2} d^*X  + l.o.t\;.$$
\end{lemma}
\begin{proof}
We differentiate Equation \eqref{tau0_vol} with respect to $\chi$:
$$ (\delta_\chi \tau_0)\rho \wedge \hat{\rho}= d(\delta_\chi \omega) \wedge\hat{\rho} + l.o.t. =d(\partial_\chi \omega \wedge\hat{\rho}) +l.o.t = d(X \lrcorner \rho \wedge\hat{\rho}) +l.o.t = -\frac{1}{2} (d^*X)\rho \wedge \hat{\rho} +l.o.t\;, $$
where $l.o.t$ are terms that depend algebraically on $\chi$ and the torsion of the $\SU(3)$-structure.
\end{proof}
We can introduce a first modification to the flow to remove one of the positive terms in the symbol.
\begin{corollary}
For $C\in \R$, consider the flow for $\sigma \in \Omega^4(M)$
\begin{equation} \label{SU3flowm2}
    \begin{cases}
        \partial_t \sigma =  \Delta_\sigma\sigma -12 \sigma  + \lie_{V(\sigma)} \sigma +d\left[(4\tau_0+ C)\hat{\rho}\right] \\
        d\sigma=0 \\
        \sigma(0)=\sigma_0\;,
    \end{cases}
\end{equation}
with $V(\sigma)$ as before and $\hat{\rho}$ the associated 3-form as usual. Then, the principal symbol of this flow satisfies  
$$ \langle \operatorname{S}_\xi (D_\sigma P)(\chi), \xi \rangle = - \abs{\xi}^2 \chi + 4 \langle \xi, JX \rangle^2 \;.$$
\end{corollary}
In particular, the question arises of whether we can further modify the flow \eqref{SU3flowm2} to obtain a parabolic flow.

\section{The Einstein--Hilbert action} \label{Einstein_stability}
Recall that Einstein metrics, solving the differential equation $\operatorname{Ric}_g =\lambda g$, are critical points of the Einstein--Hilbert action: 
\begin{equation} \label{EH-def}
    \begin{split}
    \mathcal{S}: \operatorname{Met}(M^n)  &\rightarrow \R \\
     g & \mapsto \frac{1}{n-1}\int_M  s_g - \lambda(n-2) \dvol_g \;,
    \end{split}
\end{equation}
where $\operatorname{Met}(M^n)$ is the space of metrics on $M^n$, and $s_g$ is the scalar curvature of the metric $g$.

Since nearly Kähler and nearly parallel $G_2$ manifolds are the links of Ricci flat cones, they are Einstein for $\lambda=n-1$. In particular, they are critical points of the Einstein--Hilbert action. We investigate the relation between the second variation of the Hitchin functionals and the Einstein--Hilbert functional.

For the remainder of the section, assume $(M^n,g)$ is not isometric to the round sphere. At a point $g \in \operatorname{Met}(M)$, the tangent space of $\operatorname{Met}(M)$ is identified with symmetric 2-tensors  $\Gamma\left(\Sym^2(T^*M)\right)$. As in the case of the Hitchin functionals, the functional $\mathcal{S}$ is diffeomorphism invariant. Thus, it is convenient to study variations orthogonal to the diffeomorphism orbit. We have an $L^2$-orthogonal decomposition: 
$$ \Gamma\left(\Sym^2(T^*M)\right)= \R g ~ \oplus \mathcal{C}_0^\infty(M) g~  \oplus \Gamma(TM) \oplus TT;\, $$
where the first and second terms correspond to constant rescalings and infinitesimal conformal deformations, respectively. The identification $\Gamma(TM) \rightarrow \Sym^2(TM)$ via the map $X\mapsto \lie_X g$ corresponds to the orbit of the diffeomorphism group. The term $TT$ are the traceless and transverse (tt for short) symmetric 2-tensors:
$$ TT(M,g) = \left\{ h \in \Gamma(\Sym^2(T^*M) \st \operatorname{tr}(h)=0 ,~\delta h =- \sum_i e_i \lrcorner \nabla_{e_i} h = 0 \right\}\;.$$
By Ebin’s slice theorem, this formal complement to the orbit tangent space is the tangent space to a genuine slice of the diffeomorphism orbit in a given conformal class.
\begin{theorem}{\cite[Thm 2.4 \& Thm. 2.5]{Koi79}} \label{Einstein2ndvar}
Let $(M^n,g)$ be an Einstein metric with constant $\lambda$. Then, when restricted to conformal variations, the second variation is given by 
\begin{equation}
    \delta^2 \mathcal{S}_g (f, f') =  \frac{n-2}{2}\int_M  \left\langle \Delta f - \frac{n \lambda}{n-1} f , f' \right\rangle \dvol \;.
\end{equation}
When restricted to $tt$-tensors, it is given by
\begin{equation}
    \delta^2 \mathcal{S}_g (h, h') =  - \frac{1}{n-1}\int_M \left \langle \Delta_L h - 2\lambda h, h' \right\rangle \dvol \;.
\end{equation}
for $h, h' \in TT(M,g) \subset \Gamma(\Sym^2_0 T^*M)$.
\end{theorem}
If $\lambda=n-1$, the operator $(\Delta  - n)f$ is strictly positive for $f \in \mathcal{C}^\infty_0(M)$, by Obata’s theorem \cite{Ob62}. The term $\Delta_L h$ is the Lichnerowicz Laplacian $$\Delta_L  =  \nabla^* \nabla + q(\mathcal{R})\;,$$
where $q(\mathcal{R})= \sum_{i<j} (e_i \wedge e_j )_* \left(\mathcal{R}(e_i, e_j)\right)_*$ is the standard curvature endomorphism induced by the Riemannian curvature tensor $\mathcal{R}$. One defines the co-index of an Einstein metric as the maximal subspace along which $\mathcal{S}_g |_{TT}$ is positive definite. Since the operator $\Delta_L h - 2\lambda h$ is a strongly elliptic operator, the co-index is guaranteed to be finite.

Let us study the Einstein co-index of nearly Kähler and nearly parallel $G_2$-structures.
\subsubsection*{Nearly Kähler manifolds}
We consider the case where $(M,g)$ is a nearly Kähler manifold. Since its metric cone is Ricci flat, the metric $g$ is Einstein with $\lambda=5$. By Lemma \ref{decompositionSU(3)}, we have an isomorphism 
\begin{align*}
    \Phi: \Sym^2_0 T^*M &\rightarrow \Omega^2_8 \oplus \Omega^3_{12}\\
 h = (h_+, h_- ) & \mapsto \left( I(h_+), \Upsilon(h_-) \right) \;,
\end{align*}
with $h_\pm = 1/2 \left(  h \pm JhJ \right)$ the $J$-commuting and $J$-anti-commuting parts of a traceless symmetric 2-tensor. Thus, $\overline{\Delta}=\Phi^{-1} \circ \Delta_L  \circ \Phi$ is a Laplacian-type operator on $\Omega^2_8 \oplus \Omega^3_{12}$.

The key result, due to Schwahn \cite{Sch22} (cf. \cite[Section 5]{MS11}), allows us to transform the eigenvalue problem $\mathcal{Q}-\lambda =0$  to an eigenvalue problem for the Laplacian on forms:
\begin{proposition}[{\cite[Lemma 3.2]{Sch22}}] \label{index_bound_NK}
 Let $(M^6, \omega, \rho)$ be a nearly Kähler manifold, not isometric to the round sphere. Consider the eigenspaces spaces 
 $$ \mathcal{E}(\lambda) = \left \{ \beta \in \Omega^2_{8} ~ \st  d^* \beta =0 \;, ~ \Delta \beta = \lambda \beta \right \}\;.$$
The Einstein index of $(M, g)$ is given by 
 \begin{equation} \label{einstein_index_NK}
     \operatorname{Ind}^{EH} = b^2(M) + b^3(M) + 3 \sum_{\lambda \in (0,2)} \dim \; \mathcal{E}(\lambda) + 2 \sum_{\lambda \in (2, 6)}\dim \; \mathcal{E}(\lambda) +  \sum_{\lambda \in (6,12)}
 \dim \; \mathcal{E}(\lambda)\;.
 \end{equation}
\end{proposition}
\begin{corollary}
The Einstein co-index is bounded below by the Hitchin index. 
\end{corollary}
\subsubsection*{Nearly parallel $G_2$ manifolds}
We now consider the case where $(M,g)$ is a nearly parallel $G_2$ manifold. Since its metric cone is Ricci flat, the metric $g$ is Einstein with $\lambda=6$. By Lemma \ref{formdecom7} we have an isomorphism $ i_\phi: \Gamma\left(\Sym^2_0\right) \rightarrow \Omega^3_{27}$.
The Laplacian comparison formula needed in this case is due to Alexandrov and Semmelmann.
\begin{proposition}[{\cite[Prop. 6.1]{AS12}}]
Under the map $i_\phi$, the operator
 $$ \mathcal{Q}(h) = \Delta_L h -12 h $$
 on $h \in TT(M)$ is identified with 
\begin{equation}\label{Einstein2v_G2}
        \widehat{ \mathcal{Q}} = \Delta \gamma + 2 *d \gamma - 8 \gamma 
\end{equation}
acting on $ \Omega_{TT} = \{\gamma \in \Omega^3_{27} \st \pi_7(d\gamma) =0\}$. 
\end{proposition}
The proof strategy is the same as that of nearly Kähler structures. Let us study the eigenvalue problem for $\widehat{\mathcal{Q}}$.
\begin{proposition}\label{index_bound_GS}
 Let $(M^7,g, \phi)$ be a nearly parallel $G_2$ manifold. Consider the eigenspaces spaces 
 $$ \mathcal{E}(\lambda) = \left\{ \gamma \in \Omega^3_{12}  \st   *d \gamma = \lambda \gamma \right\} ~~~~~~~~ \mathcal{F}(\lambda) = \left \{ \gamma \in \Omega^3_{12}  \st    dd^*\gamma = \lambda \gamma \right\} \;.$$
 The Einstein index of $(M, g)$ is given by
 \begin{equation} \label{einstein_index_bound_G2}
     \operatorname{Ind}^{EH} = b^3(M) + \sum_{\lambda \in (-4,0)\cup (0,2)} \dim ~\mathcal{E}(\lambda) + \sum_{\lambda \in (0,8)} \dim ~\mathcal{F}(\lambda)\;.
 \end{equation}
\end{proposition}
\begin{proof}
The operator $\widehat{\mathcal{Q}}$ commutes with the self-dual operator $*d$. Thus, we can find a common base of eigenforms. Let $\mu \in \R$ and consider the spaces $\mathcal{E}(\mu)$ and $\mathcal{F}(\mu^2)$ defined above. If $\mu\neq 0$, we have $\mathcal{E}(\mu) \subseteq \mathcal{F}(\mu^2)$, and thus are all finite-dimensional, by ellipticity of the Laplacian.

Substituting $*d\gamma = \mu \gamma$ for $\mu \neq 0$ in Equation \eqref{Einstein2v_G2}, we have that $\gamma$ is an eigenform of $\widehat{\mathcal{Q}}$ with eigenvalue $\lambda = \mu^2 + 2\mu -8$. If $\mu=0$, $\gamma$ is closed, and Equation \eqref{Einstein2v_G2} reduces to $ \Delta \gamma = dd^*\gamma = (\lambda +8) \gamma$, which concludes the proof of Equation \eqref{einstein_index_bound_G2}. From Corollary \ref{Qindex_G2}, we obtain the desired bound.
\end{proof}
\begin{remark}
    The purely topological bound $\operatorname{Ind}^{EH}\geq b^3(M)$ appeared in \cite{SWW22}. 
\end{remark}

\bibliographystyle{alpha}
\bibliography{Bibliography}

\begin{thebibliography}{FKMS97}

\bibitem[AS12]{AS12}
B.~Alexandrov and U.~Semmelmann.
\newblock {Deformations of nearly parallel $G_2$-structures}.
\newblock {\em Asian J. Math.}, 16:713--744, 2012.

\bibitem[BG08]{BG08}
C.~P. Boyer and K.~Galicki.
\newblock {\em {Sasakian geometry}}.
\newblock Oxford University Press, 2008.

\bibitem[Bry05]{Bry05}
R.~Bryant.
\newblock {Some remarks on $G_2$-structures}, 2005.
\newblock \url{https://arxiv.org/abs/math/0305124}.

\bibitem[BS89]{BS89}
R.~L. Bryant and S.~M. Salamon.
\newblock {On the construction of some complete metrics with exceptional holonomy}.
\newblock {\em Duke Mathematical Journal}, 58(3):829 -- 850, 1989.

\bibitem[BV07]{BV07}
L.~Bedulli and L.~Vezzoni.
\newblock {The Ricci tensor of $SU(3)$-manifolds}.
\newblock {\em Journal of Geometry and Physics}, 57(4):1125--1146, 2007.

\bibitem[BX11]{BX11}
R.~Bryant and F.~Xu.
\newblock {Laplacian Flow for Closed $G_2$-Structures: Short Time Behavior}, 2011.
\newblock \url{https://arxiv.org/abs/arXiv.1101.2004}.

\bibitem[Bä93]{Bar93}
C.~Bär.
\newblock {Real Killing spinors and holonomy}.
\newblock {\em Comm. Math. Phys.}, 154(3):509–52, 1993.

\bibitem[CS02]{CS02}
S.~Chiossi and S.~Salamon.
\newblock {The intrinsic torsion of $SU(3)$ and $G_2$ structures}, 2002.
\newblock Proc. conf. Differential Geometry Valencia 2001, \url{https://arxiv.org/abs/math/0202282}.

\bibitem[CV15]{CV15}
V.~Cortés and J.~J. Vásquez.
\newblock { Locally homogeneous nearly Kähler manifolds}.
\newblock {\em Ann. Global Anal. Geom.}, 48(3):269--294, 2015.

\bibitem[DS20]{SS20}
S.~Dwivedi and R.~Singhal.
\newblock {Deformation theory of nearly $G_2$ manifolds}, 2020.
\newblock \url{https://arxiv.org/abs/2007.02497}.

\bibitem[FH17]{FH17}
L.~Foscolo and M.~Haskins.
\newblock {New $G_2$-holonomy cones and exotic nearly Kähler structures on $S^6$ and $S^3 \times S^3$}.
\newblock {\em Annals of Mathematics}, 185:59 --130, 2017.

\bibitem[FKMS97]{FKMS97}
T.~Friedrich, I.~Kath, A.~Moroianu, and U.~Semmelmann.
\newblock {On nearly parallel G2-structures}.
\newblock {\em Journal of Geometry and Physics}, 23(3):259--286, 1997.

\bibitem[Fos17]{Fos17}
L.~Foscolo.
\newblock {Deformation theory of nearly Kähler manifolds}.
\newblock {\em Journal of the London Math. Soc.}, 95:586--612, 2017.

\bibitem[GH80]{GH80}
A.~Gray and L.~Hervella.
\newblock {The sixteen classes of almost Hermitian manifolds and their linear invariants}.
\newblock {\em Annali di Matematica Pura ed Applicata}, 123:35--58, 1980.

\bibitem[Gra70]{Gray70}
A.~Gray.
\newblock {Nearly Kähler manifolds}.
\newblock {\em J. Differential Geometry}, 4:283--309, 1970.

\bibitem[Gri13]{Gri13}
S.~Grigorian.
\newblock {Short-time behaviour of a modified Laplacian coflow of $G_2$-structures}.
\newblock {\em Advances in Mathematics}, 248:378--415, 2013.

\bibitem[Hit00]{Hitchin00}
N.~Hitchin.
\newblock {The geometry of three-forms in six and seven dimensions}, 2000.
\newblock \url{https://arxiv.org/abs/math/0010054}.

\bibitem[Hit01]{Hitchin01}
N.~Hitchin.
\newblock Stable forms and special metrics, 2001.
\newblock \url{https://arxiv.org/abs/math/0107101}.

\bibitem[Joy00]{Joy00}
D.~Joyce.
\newblock {\em {Compact Manifolds with Special Holonomy}}.
\newblock Oxford University Press, 2000.

\bibitem[Joy03]{Joy032}
D.~Joyce.
\newblock {Special Lagrangian submanifolds with isolated conical singularities}, 2003.
\newblock \url{https://arxiv.org/abs/math/0211295}.

\bibitem[Kar08]{Kar08}
S.~Karigiannis.
\newblock {Some Notes on $G_2$ and $Spin(7)$ Geometry}, 2008.
\newblock \url{https://arxiv.org/abs/math/0608618}.

\bibitem[KL20]{LK20}
S.~Karigiannis and J.D. Lotay.
\newblock {Deformation theory of $G_2$ conifolds}.
\newblock {\em Communications in Analysis and Geometry}, 28:1057--1210, 2020.

\bibitem[KMT12]{KMT12}
S.~Karigiannis, B.~McKay, and M.~Tsui.
\newblock {Soliton solutions for the Laplacian co-flow of some $G_2$-structures with symmetry}.
\newblock {\em Differential Geometry and its Applications}, 30(4):318--333, 2012.

\bibitem[Koi79]{Koi79}
N.~Koiso.
\newblock {On the second derivative of the total scalar curvature}.
\newblock {\em Osaka Journal of Mathematics}, 16(2):413--421, 1979.

\bibitem[Leh21]{Leh21}
F.~Lehmann.
\newblock {Deformations of asymptotically conical $Spin(7)$-manifolds}, 2021.
\newblock \url{https://arxiv.org/abs/2101.10310}.

\bibitem[Lot07]{Lot07}
J.D. Lotay.
\newblock {Coassociative 4-folds with conical singularities}.
\newblock {\em Communications in Analysis and Geometry}, 15:891--946, 2007.

\bibitem[LT05]{LT05}
D.~Lüst and D.~Tsimpis.
\newblock {Supersymmetric AdS4 compactifications of IIA supergravity}.
\newblock {\em Journal of High Energy Physics}, 2005(02), 2005.

\bibitem[MNS08]{MNS08}
A.~Moroianu, P.A. Nagy, and U.~Semmelmann.
\newblock {Deformations of nearly Kähler structures}.
\newblock {\em Pacific Journal of Mathematics}, 235:57--72, 2008.

\bibitem[MS10]{MS10}
A.~Moroianu and U.~Semmelmann.
\newblock {The Hermitian Laplace Operator on Nearly Kähler Manifolds}.
\newblock {\em Communications in Mathematical Physics}, 294:251--272, 2010.

\bibitem[MS11]{MS11}
A.~Moroianu and U.~Semmelmann.
\newblock {Infinitesimal Einstein deformations of nearly Kähler metrics}.
\newblock {\em Transactions of the American Mathematical Society}, 363(6):3057--3069, 2011.

\bibitem[Nag01]{Nagy01}
P.A. Nagy.
\newblock {\em {A principle of separations of variables for the spectrum of Laplacian acting on forms and applications}}.
\newblock PhD thesis, Universit of Savoie, 2001.

\bibitem[Nag02]{Nagy02}
P.A. Nagy.
\newblock {Nearly Kähler geometry and Riemannian foliations}.
\newblock {\em Asian J. Math.}, 3:481--504, 2002.

\bibitem[Nor08]{Nor08}
J.~Nordström.
\newblock {\em {Deformations and glueing of asymptotically cylindrical manifolds with exceptional holonomy}}.
\newblock PhD thesis, Cambridge, King's College, 2008.

\bibitem[NS21]{NS21}
P.A. Nagy and U.~Semmelmann.
\newblock {Deformations of nearly $G_2$-structures}.
\newblock {\em J.London Math.Soc.}, 104:1795--1811, 2021.

\bibitem[NS23]{NS23}
P.A. Nagy and U.~Semmelmann.
\newblock {The $G_2$ geometry of $3$-Sasaki structures}, 2023.
\newblock \url{https://arxiv.org/abs/2101.04494}.

\bibitem[Oba62]{Ob62}
M.~Obata.
\newblock {Certain conditions for a Riemannian manifold to be isometric with a sphere}.
\newblock {\em Journal of The Mathematical Society of Japan}, 14:333--340, 1962.

\bibitem[Pod21]{Pode21}
F.~Podestà.
\newblock {Nearly Parallel $G_2$-structures with Large Symmetry Group}.
\newblock {\em Canadian Journal of Mathematics}, 73(2):339--359, 2021.

\bibitem[RL82]{HL82}
R.Harvey and H.B. Lawson.
\newblock {Calibrated geometries}.
\newblock {\em Acta Math.}, 148:47--157, 1982.

\bibitem[Sch22]{Sch22}
P.~Schwahn.
\newblock {Coindex and Rigidity of Einstein Metrics on Homogeneous Gray Manifolds}.
\newblock {\em The Journal of Geometric Analysis}, 32, 2022.

\bibitem[Sin23]{singhal23}
R.~Singhal.
\newblock {Nearly half-flat $\rm{SU}(3)$-structures on $S^3\times S^3$}, 2023.
\newblock \url{https://arxiv.org/abs/2310.11233}.

\bibitem[{Sol}24]{ESF24c}
E.~{Solé-Farré}.
\newblock {The Hitchin index in cohomogeneity one nearly Kähler structures}, 2024.
\newblock \url{https://arxiv.org/abs/2410.21106}.

\bibitem[Spa11]{Sparks11}
J.~Sparks.
\newblock {Sasaki-Einstein manifolds}.
\newblock {\em Surv. Diff. Geom.}, 16:265--324, 2011.

\bibitem[SW17]{SW10}
D.~Salamon and T.~Walpuski.
\newblock {Notes on the octonions}, 2017.
\newblock \url{https://arxiv.org/abs/1005.2820}.

\bibitem[SWW22]{SWW22}
U.~Semmelmann, C.~Wang, and M.~Wang.
\newblock {Linear instability of Sasaki Einstein and nearly parallel $G_2$ manifolds}.
\newblock {\em International Journal of Mathematics}, 33(6), 2022.

\end{thebibliography}

\end{document}